\documentclass[11pt]{article}

\usepackage{amsmath, amsfonts, amssymb, amsthm, url}
\usepackage{tikz}
\usepackage{algorithmic}
\usepackage{algorithm}
\usepackage{graphicx}
\usepackage{hyperref}

\newcommand{\FOREACH}{\FORALL}

\newcommand{\calC}{\mathcal{C}}
\newcommand{\calD}{\mathcal{D}}

\newcommand{\calQ}{\mathcal{Q}}
\newcommand{\calR}{\mathcal{R}}

\newcommand{\FF}{\mathbb{F}}

\newcommand{\rdC}{\widetilde{C}}
\newcommand{\rdE}{\widetilde{E}}
\newcommand{\rdP}{\widetilde{P}}
\newcommand{\rdKh}{\widetilde{Kh}}

\newcommand{\qC}{\overline{C}}
\newcommand{\qE}{\overline{E}}
\newcommand{\qKh}{\overline{Kh}}

\newcommand{\qbP}{{\overline \bP}}
\newcommand{\undecC}{\overline{\calC}}

\newcommand{\HFhat}{\widehat{HF}}

\DeclareMathOperator{\rank}{rank}
\newcommand{\colonns}{\colon\negmedspace}
\newcommand{\bd}{{\bf d}}
\newcommand{\bt}{{\bf t}}
\newcommand{\bF}{{\bf F}}
\newcommand{\bP}{{\bf P}}

\newtheorem{theorem}{Theorem}[section]
\newtheorem{conjecture}[theorem]{Conjecture}
\newtheorem{lemma}[theorem]{Lemma}

\newtheorem{proposition}[theorem]{Proposition}
\newtheorem{definition}[theorem]{Definition}

\title{Computations of Szab\'o's Geometric Spectral Sequence in Khovanov Homology}
\author{Cotton Seed}

\begin{document}

\maketitle

\begin{abstract}
  Szab\'o recently introduced a combinatorially-defined spectral
  sequence in Khovanov homology \cite{szaboss}.  After reviewing its
  construction and explaining our methodology for computing it, we
  present results of computations of the spectral sequence.  Based on
  these computations, we make a number of conjectures concerning the
  structure of the spectral sequence, and towards those conjectures,
  we prove some propositions.
\end{abstract}

\section{Introduction}

Khovanov homology associates to a knot or link $L$ in $S^3$ a bigraded
abelian group $Kh(L)$ which categorifies the unnormalized Jones
polynomial of $L$ \cite{khovanov}, \cite{bncat}.  The hat variant of
Heegaard Floer homology associates to a closed, oriented $3$-manifold
$Y$ an abelian group $\HFhat(Y)$.  In \cite{oszdouble}, Ozsv\'ath and
Szab\'o constructed a spectral sequence $E_{HF}^k = E^k_{HF}(L)$ from
the the reduced Khovanov homology $\rdKh(L)$ of $L$ to the Heegaard
Floer homology $\HFhat(-\Sigma(L))$ with $\FF_2$ coefficients
throughout, where $\Sigma(L)$ denotes the double cover of $S^3$
branched along $L$.  Baldwin \cite{baldwin} showed the higher pages of
this spectral sequence are invariants of $L$.  Let $\calD$ be a planar
diagram for $L$.  The construction of the $E_{HF}^k$ associates terms
in a differential of a filtered chain complex determined by counting
certain pseudoholomorphic polygons to higher faces of the cube of
resolutions of $\calD$.

Bloom constructed a spectral sequence in Khovanov homology in the
context of monopole Floer homology \cite{bloomss} and Kronheimer and
Mrowka in the context of instanton Floer homology \cite{kronmorwss}.

Lipshitz, Ozsv\'ath, and Thurston, using techniques from bordered
Floer homology \cite{lotbordered}, gave an algorithm to compute
$\HFhat(Y)$ \cite{lothfhat} and, more generally, the spectral sequence
$E_{HF}^k$ \cite{lotss1}, \cite{lotss2}.  Lipshitz developed a program
in Sage \cite{sage} to compute $\HFhat(\Sigma(L))$ \cite{lotbordprog}.
Zhan ported this program to C++ and completed support for computing
the full spectral sequence $E_{HF}^k$ \cite{zhan}.

In \cite{szaboss}, Szab\'o introduced a combinatorially-defined
spectral sequence $E^k(L)$ whose $E^2$-page is the the Khovanov
homology $Kh(L)$ and proved it is an invariant of $L$ for $k\ge 2$.
Like Khovanov homology, the spectral sequence admits a reduced variant
$\rdE(L, c)$ associated to a link with distinguished component $c$.
Like $E_{HF}^k$, the spectral sequence is constructed by associating
to higher faces of the cube of resolutions certain terms of a filtered
chain complex differential.  However, whereas the terms in the
Heegaard Floer, monopole and instanton spectral sequences involve
certain analytic considerations that make computation difficult,
Szab\'o's construction is purely combinatorial.  The author has
developed software to compute the spectral sequence $E^k$ and the aim
of this paper is to describe the results of those computations.

Zhan computed $\HFhat(-\Sigma(K))$ for all knots with at most 14
crossings except for the 11 knots $14n5631$--$5635$, $14n5637$,
$14n5643$--$5645$, $14n6285$ and $14n6302$.  He also computed the
spectral sequence $E_{HF}^k(K)$ for all knots with at most 12
crossings.  In each case, our computation of $\rdE^\infty$ and
$\rdE^k$ matched his, respectively.  This evidence, along with the
further computations and conjectures presented in Section 4, suggest
the following two conjectures:
\begin{conjecture}\label{conj1}
  Let $K$ be a knot in $S^3$.  The spectral sequence $\rdE^k(K)$
  collapsed onto the homological grading and $E_{HF}^k(K)$ are
  isomorphic as graded vector spaces.
\end{conjecture}
\noindent which would imply the weaker conjecture:
\begin{conjecture}\label{conj2}
  Let $K$ be a knot in $S^3$.  The rank of $\rdE^\infty(K)$ is equal
  to the rank of $E_{HF}^\infty(-\Sigma(K)) = \HFhat(-\Sigma(K))$.
\end{conjecture}
It is natural to expect these conjectures to hold for links also;
however, at this point we have only limited computational evidence for
the more general case.

\bigskip

\noindent {\bf Organization.}  In Section 2 we introduce notation and
briefly review the construction of Szab\'o's geometric spectral
sequence $E^k$.  In Section 3, we describe and give pseudocode for the
algorithm we use to compute the spectral sequence.

In Section 4, we describe the results of the computations, make a
number of conjectures and prove several propositions concerning the
structure of $E^k$.  In particular, our main result is that the two
ways of defining the reduced spectral sequence, as a sub- or quotient
complex, agree.  Specifically, let $C(\calD, \bt)$ denote the filtered
chain complex which induces the spectral sequence $E^k(L)$.  Let $P$
be a point on $\calD$ which corresponds to a distinguished component
$c$ of $L$.  There is a subcomplex $C(\calD, \bt, P)$ of $C(\calD,
\bt)$ associated to $P$ which induces $\rdE^k(L, c)$.  Let $\qC(\calD,
\bt, P)$ denote the quotient complex $C(\calD, \bt)/C(\calD, \bt, P)$.
The quotient complex induces a spectral sequence $\qE^k(L, c)$.  We
prove the following proposition.
\begin{proposition}\label{babytwinarrows}
  Let $(L, c)$ be a link with distinguished component.  Then we have
  \[ \rdE^k(L, c)\cong \qE^k(L, c) \]
  for all $k \ge 2$.
\end{proposition}

Finally, in Appendix A we give the the reduced Poincar\'e polynomial
of the spectral sequence for knots with at most 11 crossings and small
torus knots with nontrivial higher differentials.

\bigskip

\noindent {\bf Acknowledgments.}  I wish to thank Zolt\'an Szab\'o for
suggesting this project and for his patience and encouragement.  I
also wish to thank Zolt\'an and Bohua Zhan for sharing various
computations of Heegaard Floer homology while searching for
counterexamples to Conjectures~\ref{conj1} and \ref{conj2}.  Finally,
I also wish to thank Sam Lewallen and William Cavendish for helpful
discussions.

\section{Szab\'o's Geometric Spectral Sequence}

We briefly review the construction of Szab\'o's geometric spectral
sequence \cite{szaboss}.  Let $L$ be an oriented knot or link in $S^3$
and $\calD$ a diagram for $L$ in $S^2$ with $n$ crossings.  The
spectral sequence $E^k(L)$ is constructed by defining higher
differentials on the Khovanov chain complex associated to the faces of
the cube of resolutions.  These higher differentials depend on a
certain choice of orientation at each crossing, denoted $\bt$.  The
spectral sequence $E^k(L)$ is induced from a filtered chain complex
$C(\calD, \bt) = (C_\calD, \bd(\bt))$.  The group $C_\calD$ is the
same group underlying the Khovanov complex.

\begin{figure}
  \begin{center}
    \includegraphics{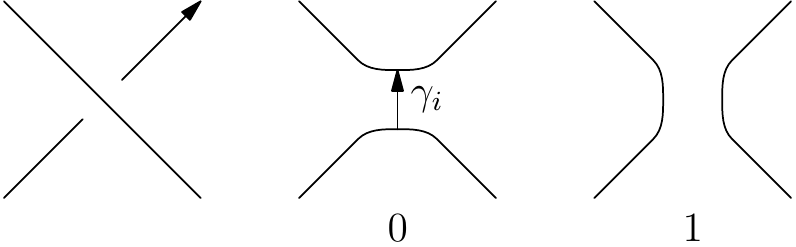}
  \end{center}
  \caption{The $0$ and $1$ resolutions associated to a
    crossing.\label{resolutionsfig}}
\end{figure}

We begin by recalling the definition of $C_\calD$.  Associated to a
crossing are to resolutions: $0$ and $1$, see
Figure~\ref{resolutionsfig}.  A resolution of $\calD$ is a choice of
resolution for each crossing.  A resolution gives a collection of
disjoint, embedded circles in the 2-sphere.  By ordering the crossings
of $\calD$, we can identify the set $\calR$ of resolutions of $\calD$
with $\{0, 1\}^n$.  Let $x$ be a circle in the sphere.  Let $V(x)$
denote the vector space over $\FF_2$ generated by $1$ and $x$.  Define
the grading $gr$ on $V(x)$ by
\begin{align*}
  gr(1) = 1 \\
  gr(x) = -1.
\end{align*}
If $S = (x_1, \dotsc, x_t)$ is a collection of circles in
the sphere, define
\[ V(S) = \bigotimes_{i=1}^t V(x_i). \]
The grading $gr$ on $V(x_i)$ induces a grading on $V(S)$.  The group
$C_\calD$ is given by
\[ C_\calD = \bigoplus_{I\in \calR} V(I). \]
The group $C_\calD$ carries a bigrading.  Let $x$ be a monomial in
$V(I)$.  The homological grading of $x$ is given by
\[ h(x) = |I| - n_- \]
and its quantum grading is
\[ q(x) = gr(x) + |I| + n_+ - 2n_-, \]
where $|I|$ denotes the number of 1 digits in $I$, and $n_-$ and $n_+$
are the number of negative and positive crossings of $\calD$,
respectively.  There is also a $\delta$ grading given by
\[ \delta(x) = q(x) - 2h(x). \]

We now recall the construction of the differential $\bd(\bt)$.  In the
$0$-resolution, there is an arc between the segments of the resolution
such that surgery along this arc gives the $1$-resolution.  Let $\bt$
denote a choice of orientation of the $0$-resolution surgery arc for
each crossing.  The pair $(\calD, \bt)$ is called a decorated diagram.

The differential is defined in terms of configurations.  A
$k$-dimensional configuration $\calC = (x_1, \dotsc, x_t, \gamma_1,
\dotsc, \gamma_k)$ is a collection of embedded circles $(x_1, \dotsc,
x_t)$ in $S^2$ together with $k$ embedded, oriented arcs $\gamma_1,
\dotsc, \gamma_k$ such that the circles and the interior of the arcs
are all disjoint and the endpoints of the arcs lie on the circles.
Recall the following operations on configurations:
\begin{itemize}
\item {\em undecorated configuration.}  The undecorated configuration
  $\bar{\calC}$ is obtained from $\calC$ by forgetting the orientation of
  the arcs.

\item {\em dual configuration.}  The dual configuration $\calC^* = (y_1,
  \dotsc, y_s, \gamma_1^*, \dotsc, \gamma_k^*)$ is the configuration
  obtained from $\calC$ by performing surgery along the arcs
  $\gamma_i$.  The dual arcs $\gamma_i^*$ are obtained by rotating the
  arcs $\gamma_i$ counterclockwise by 90 degrees.

\item {\em reverse configuration.} The reverse configuration $r(\calC)$ is
  obtained from $\calC$ by reversing the orientation of the arcs
  $\gamma_i$.

\item {\em mirror configuration.}  The mirror configuration $m(\calC)$
  is obtained from $\calC$ by reversing the orientation of the ambient
  2-sphere.

\end{itemize}

The circles $x_i$ are called the starting circles of $\calC$.  The
circles $y_i$ of $\calC^*$ are called the ending circles of $\calC$.
Set 
\[ V_0(\calC) = V(x_1, \dotsc, x_t), \qquad V_1(\calC) = V(y_1, \dotsc, y_s). \]

Let $P(\calC) = (x_{i_1}, \dotsc, x_{i_k})$ be the collection of
circles of $\calC$ which are disjoint from the arcs; they are called
the passive circles of $\calC$.  Let $\calC_0$ denote the
configuration obtained by deleting the passive circles; this is called
the active part of $\calC$.  A configuration with no passive circles
is called purely active.  There are decompositions
\[ V_0(\calC) = V_0(\calC_0)\otimes P(\calC), \qquad V_1(\calC) = V_1(\calC_0)\otimes P(\calC). \]

Let $(I, J)$ be a $k$-face of the cube of resolutions.  The
$k$-dimensional configuration $\calC(I, J, \bt) = (x_1, \dotsc, x_t,
\gamma_1, \dotsc, \gamma_k)$ consists of the circles $(x_1, \dotsc,
x_t)$ of $I$ together with the $0$-resolution surgery arcs
$\gamma_{i_1}, \dotsc, \gamma_{i_k}$ corresponding to the coordinates
$i_1, \dotsc, i_k$ where $I$ and $J$ differ and with orientation given
by $\bt$.

Given a map $F_\calC : V_0(\calC)\to V_1(\calC)$ for each
configuration $\calC$, there is a induced map $\bF(\bt)$ on $C_\calD$
given by
\[ \bF(\bt) = \sum_{i=1}^n \bF^k(\bt), \]
where
\[ \bF^k(\bt) = \sum_\text{$k$-faces (I, J)} F_{I, J, \bt} \]
and
\[ F_{I, J, \bt} = F_{\calC(I, J, \bt)}. \]

Next, we define a series of properties of maps the form $F_\calC$.

\begin{definition}[Extension Formula]
  If $F_\calC$ satisfies the formula
  \[ F_\calC(a\cdot v) = F_{\calC_0}(a)\cdot v \]
  for $v\in P(\calC)$ and $a\in V_0(\calC_0)$, then we say it
  satisfies the extension formula.
\end{definition}

In particular, if $F_\calC$ satisfies the extension formula, then it
is determined by its value on purely active configurations.

A configuration $\calC$ is disconnected if the graph with vertices
circles of $\calC_0$ and edges arcs $\gamma_i$ is disconnected.

\begin{definition}[Disconnected rule]
  If $\calC$ is a disconnected configuration and
  \[ F_\calC = 0, \]
  then we say $F_\calC$ satisfies the disconnected rule
\end{definition}

\begin{definition}[Naturality rule]
  Let $\calC$ and $\calC'$ be configurations such that there exists an
  orientation preserving diffeomorphism of $S^2$ carrying $\calC$ to
  $\calC'$.  The diffeomorphism induces an identification between
  $V_0(\calC) = V_0(\calC')$ and $V_1(\calC) = V_1(\calC')$.  If,
  under these identifications, we have
  \[ F_\calC = F_{\calC'}, \]
  then we say $F_\calC$ satisfies the naturality rule.
\end{definition}

\begin{definition}[Filtration rule]
  Let $\calC$ be a configuration and $P$ a point on a starting circle
  away from the endpoints of the arcs.  Let $x(P)$ denote the starting
  circle containing $P$ and $y(P)$ the ending circle containing $P$.
  Let $a\in V_0(\calC)$ and $b\in V_1(\calC)$ be monomials.  If
  $F_\calC(a, b) = 1$ and $x(P)$ divides $a$ implies $y(P)$ divides
  $b$, then we say $F_\calC$ satisfies the filtration rule.
\end{definition}

\begin{definition}[Grading rule]
  Let $\calC$ be a $k$-dimensional configuration, $a\in V_0(\calC)$
  and $b\in V_1(\calC)$ monomials.  If $F_C(a, b) = 1$ implies
  \[ gr(b) - gr(a) = k - 2, \]
  then we say $F_\calC$ satisfies the grading rule.
\end{definition}

\begin{figure}
  \begin{center}
    \includegraphics{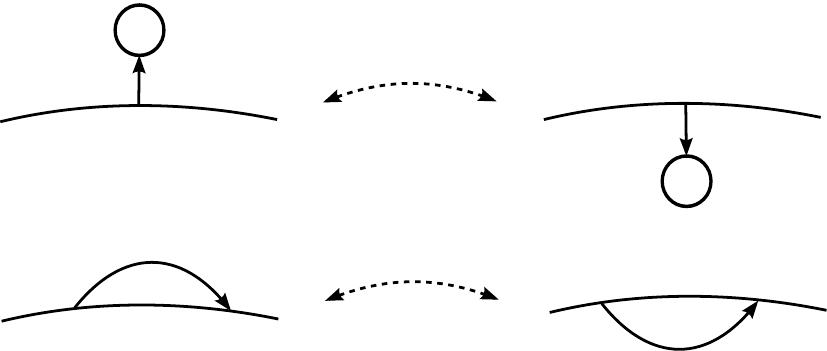}
  \end{center}
  \caption{The two rotation moves.\label{rotfig}}
\end{figure}

\begin{definition}[Rotation rule]
  Suppose two configurations $\calC$ and $\calC'$ differ by one of the
  rotation moves in Figure~\ref{rotfig}.  There is a natural
  identification $V_0(\calC) = V_0(\calC')$ and $V_1(\calC) =
  V_1(\calC')$.  If $F_\calC = F_{\calC'}$, then we say $F_\calC$
  satisfies the rotation rule.
\end{definition}

\begin{figure}
  \begin{center}
    \includegraphics{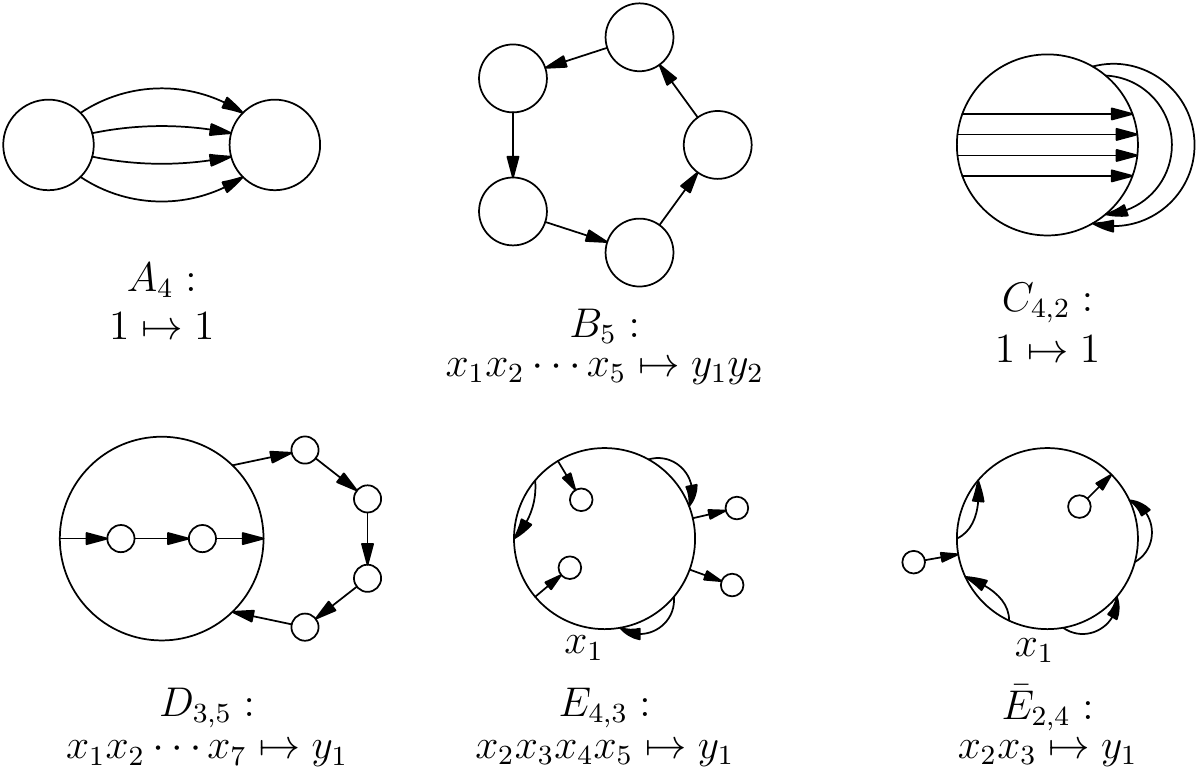}
  \end{center}
  \caption{Examples of the five families of
    configurations.\label{familiesfig}}
\end{figure}

The map $d_\calC$ defining the differential $\bd(\bt)$ is given by
five families $A_k$, $B_k$, $C_{p,q}$, $D_{p,q}$ and $E_{p,q}$, $p + q
= k$, of $k$-dimensional configurations for which $d_\calC\neq 0$.
Examples of those five families are given in Figure~\ref{familiesfig}.
For the precise definition of $d_\calC$, we refer the reader to
\cite{szaboss}.  The map $d_\calC$ satisfies all of the above
properties.  Note, $d_\calC$ satisfies several additional properties,
including the conjugation and duality rules, see \cite{szaboss}.  As
these properties are not used in the sequel, we will not recall their
definition here.

The grading rule implies that $\bd^k(\bt)$ has homogeneous degree $(k,
2k - 2)$ and $\delta$ degree $-2$.  Thus, the total differential
$\bd(\bt)$ has $\delta$ degree $-2$.  The homological grading induces
a filtration on $C_\calD$.

\begin{theorem}[Szab\'o \cite{szaboss}]
  The map $\bd(\bt)$ is a differential, that is,
  \[ \bd(\bt)\circ \bd(\bt) = 0. \]
  The spectral sequence $E^k(L)$ induced from the filtration coming
  from the homological grading is an invariant of the oriented link
  $L$ for $k\ge 2$.
\end{theorem}

Finally, we recall two related constructions: the reduced version and
the mirror version.  Let $L$ be a link with distinguished component
$c$.  Let $P$ be a point of $\calD$ on the distinguished component
away from the crossings.  For a resolution $I$, let $x(P)$ denote the
circle meeting $P$.  Let $C(\calD, \bt, P)$ be the subcomplex of
$C(\calD, \bt)$ generated by monomials divisible by $x(P)$.  It is a
subcomplex by the filtration rule.  It is traditional to shift the
quantum grading up by $1$ in the reduced subcomplex.  Szab\'o showed
the reduced spectral sequence $\widetilde{E}^k(L, c)$ for $k\ge 2$ is
an invariant of the pair $(L, c)$.

\begin{figure}
  \begin{center}
    \includegraphics{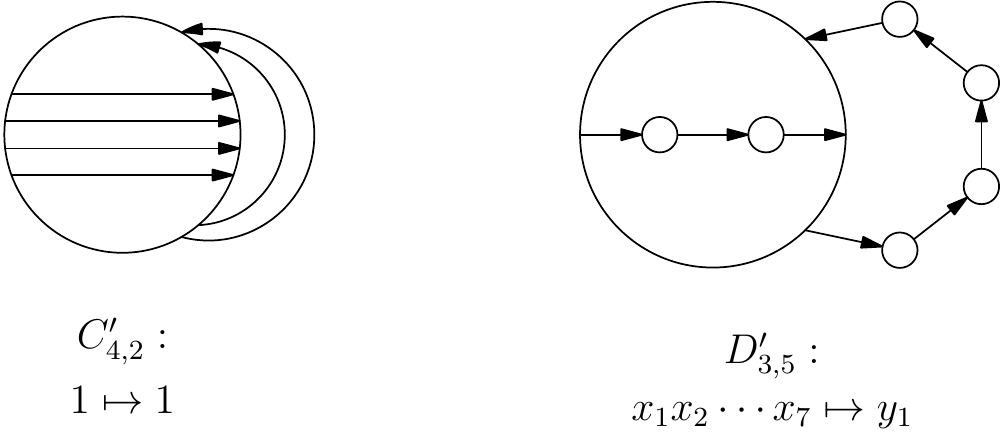}
  \end{center}
  \caption{Examples of mirror configurations of type $C'$ and $D'$.\label{altdfig}}
\end{figure}

There is an implicit choice of orientation in the configurations of
type $C$ and $D$.  Saab\'o notes there is an alternate construction
${\bf d'}$ of the differential where the map $d'_\calC$ is given by
\[ d'_\calC = d_{m(\calC)}. \]
This amounts to choosing the mirror configuration types $C'$ and $D'$
in the definition of $d'_\calC$, see Figure~\ref{altdfig}.  The other
configuration types are invariant under taking the mirror.  We denote
by ${E'}^k(L)$ the mirror spectral sequence coming from the mirror complex
$C'(\calD, \bt) = (C_\calD, {\bf d'}(\bt))$.

\section{Computing The Spectral Sequence}

We now describe the algorithm used for computing the spectral
sequence.  We begin with an decorated planar link diagram $(\calD,
\bt)$.

\begin{algorithm}
\caption{Algorithm to construct the chain complex $(C_\calD, \bd(\bt))$.}
\label{dalg}
\begin{algorithmic}
\STATE $\bd(\bt) \leftarrow 0$
\FOREACH{face $(I, J)$ of $\calR$}
\STATE $\calC \leftarrow \calC(I, J, \bt)$
\IF{$\calC_0$ has type $A, B, C, D$ or $E$}
\FOREACH{generator $a$ of $V(\calC_0)$ and $p$ of $P(\calC)$}
\STATE $\bd(\bt)(a\cdot p) \leftarrow \bd(\bt)(a\cdot p) + d_{\calC_0}(a)\cdot p$
\ENDFOR
\ENDIF
\ENDFOR
\RETURN $\bd(\bt)$
\end{algorithmic}
\end{algorithm}

The first step is to build the chain complex $C(\calD, \bt)$.  As with
Khovanov homology, the combinatorial description of the chain group
$C_\calD$ and differential $\bd(\bt)$ is amenable to direct
computation.  In our implementation, we naively follow the
combinatorial description.  The generators of $C_\calD$ are
represented by pairs $I\colonns m$ with $I\in {\bf 2}^n$ and $m\in
{\bf 2}^t$ where $t = t(I)$ is the number of circles in the resolution
$I$.  To calculate $\bd(\bt)$ we simply sum the contributions
$d_{\calC(I, J)}$ for each face $(I, J)$ of $\calR$.  Pseudocode is
given in Algorithm~\ref{dalg}.

After building the chain complex, the next step is compute the
spectral sequence.  We use repeated application of the cancellation
lemma as outlined by Baldwin in \cite[Section 4]{baldwin}.

\begin{lemma}[Cancellation Lemma]
  Let $(C, d)$ be a chain complex freely generated by $\{x_i\}$.
  Let $d(x_i, x_j)$ denote the coefficient of $x_j$ in $d(x_i)$.
  Suppose $d(x_k, x_\ell) = 1$.  Let $(C', d')$ be the complex where
  $C'$ is generated by $\{x_i | i\neq k, \ell\}$ and the differential
  $d'$ given by
  \[ d'(x_i) = d(x_i) + d(x_i,x_\ell)d(x_k). \]
  The $(C, d)$ is chain homotopy equivalent to $(C', d')$.
\end{lemma}
We say that $(C', d')$ is obtained from $(C, d)$ by canceling the term
$d(x_k, x_\ell)$.  The cancellation lemma admits a refinement for
filtered complexes.  This refinement, together with the mapping lemma
for spectral sequences, establishes the following process for
computing the spectral sequence.  The pair $(E^0, d^0)$ is simply
$(C_\calD, \bd^0(\bt))$.  Then we cancel the terms of $\bd(\bt)$ that
preserve the homological grading, to obtain a new complex which, by
abuse of notation, we also call $(C_\calD, \bd(\bt))$.  The pair
$(E^1, d^1)$ is then $(C_\calD, \bd^1(\bt))$, where $\bd^1(\bt)$
denotes the terms of $\bd(\bt)$ which increase the homological grading
by $1$.  Then we cancel the terms of $\bd(\bt)$ which shift the
homological grading by $1$, and $(E^2, \bd^2(\bt))$ is $(C_\calD,
\bd^2(\bd))$.  This process terminates when all the terms of the
differential are canceled and $\bd(\bt) = 0$.  This will always
happen since the homological degree of $C_\calD$ has bounded support.
Pseudocode for the process of canceling terms to compute the spectral
sequence is given in Algorithm~\ref{ssalg}.

\begin{algorithm}
\caption{Algorithm to compute the spectral sequence $E^k$ associated
  to the filtered chain complex $(C_\calD, \bd(\bt))$.}
\label{ssalg}
\begin{algorithmic}
\STATE $i \leftarrow 0$
\WHILE{$\bd(\bt)\neq 0$}
\STATE $(E^i, d^i) \leftarrow (C_\calD, \bd^i(\bt))$
\WHILE{$\bd(\bt)(x_k, x_\ell) = 1$ for some $k, \ell$ with $h(x_l) - h(x_k) = i$}
\STATE cancel $\bd(\bt)(x_k, x_l)$ in $(C_\calD, \bd(\bt))$
\ENDWHILE
\STATE $i\leftarrow i + 1$
\ENDWHILE
\RETURN $E^k$
\end{algorithmic}
\end{algorithm}

Both the time and space complexity of the naive algorithm are
exponential in the number of crossings of the diagram.  In practice,
it is feasible to compute the $E^k$ for knots with 18--19 crossings on
a computer with 12Gb of RAM.  The author has developed a second
program to compute $E^k$ based on a partial construction of the
spectral sequence for tangles obtained by adapting algebraic
techniques from bordered Floer homology \cite{lotslicing}.  We plan to
describe this construction in a subsequent paper.

In \cite{bntangle}, \cite{bnfast}, Bar-Natan introduced a fast
divide-and-conquer algorithm for computing Khovanov homology.
Although there is no analysis of its algorithmic complexity, its
running time appears to depend heavily on the girth of the knot, with
girth $14$ the upper end of the feasible range \cite{freedmanetal}.
It would also be interesting to see if Bar-Natan's formulation could
be extended to yield a fast algorithm for computing Szab\'o's spectral
sequence.

\section{Results and Conjectures}

In this section, we present results of computations of the spectral
sequence $E^k$.  Unless stated otherwise, the conjectures in this
section have been verified on all primes links with $12$ or fewer
crossings, all prime knots with $14$ or fewer crossings, and all torus
knots $T_{p,q}$ where $(p - 1)q \le 16$.  We used knot data from two
sources.  We extracted the planar diagram (PD) description of the
Rolfsen knot tables from Bar-Natan's {\tt KnotTheory`} package
\cite{bnknotth}.  In addition, we used the HTW knot tables
\cite{htwknots} and the Thistlethwaite link (MT) tables and from
SnapPy \cite{snappy}.  The HTW knot tables and MT link tables are
encoded with Dowker-Thistlethwaite (DT) codes.  The HTW tables
included all prime knots through 16 crossings and the MT tables
include all prime links through 14 crossings.

In order to simplify presentation of the results of the computations,
we begin with the following two conjectures.

\begin{conjecture}
  The reduced theory does not depend on the choice of distinguished
  component, that is, if $L$ is a link and $c$ and $c'$ are components
  of $L$, then
  \[ \rdE^k(L, c) \cong \rdE^k(L, c') \]
  for $k\ge 2$.
\end{conjecture}

Thus we will write $\rdE(L)$ for the reduced spectral sequence of a
link $L$.

\begin{conjecture}[Twin Arrows]
  Let $L$ be a link.  For $k\ge 2$, the page $E^k(L)$ is isomorphic to
  two copies of $\rdE^k(L)$.  Specifically,
  \[ E^k(L) \cong \rdE^k(L)\{-1\} \oplus \rdE^k(L)\{1\}. \]
\end{conjecture}

In Appendix A, we give the Poincar\'e polynomials $\rdP(q, t)$ for the
reduced spectral sequence $\rdE^k$ for all knots with $11$ or fewer
crossings and all torus links $T_{p,q}$ with $(p-1)q \le 16$ for which
the spectral sequence has nontrivial higher differentials.

Recall from the introduction, $\qC(\calD, \bt, P)$ is the quotient
complex 
\[ C(\calD, \bt)/\rdC(\calD, \bt, P) \]
and $\qE^k(L, c)$ is the spectral sequence induced from $\qC(\calD,
\bt, P)$.  Towards the twin arrows conjecture, we prove our main
result, Proposition~\ref{babytwinarrows}.

\begin{proof}[Proof of Proposition~\ref{babytwinarrows}]
We will define a chain map
\[ \bP(\bt) : C(\calD, \bt)\to C(\calD, \bt) \]
that induces the desired chain homotopy equivalence.  The construction
of $\bP(\bt)$ and proof that it is a chain map closely mimics the the
construction of $\bd(\bt)$ and proof that it is a differential in
\cite{szaboss}.  We define a map
\[ P_\calC : V_0(\calC)\to V_1(\calC) \]
for each configuration $\calC$.  The map $P_\calC$ satisfies the
extension formula, the disconnected rule, the naturality rule and the
rotation rule.  The map $P_\calC$ satisfies an additional property
related to the point $P$.  If $a\in V_0(\calC)$ and $x(P)$ divides
$a$, then $P_\calC(a) = 0$.  If $P_\calC(a)\neq 0$ for some $a\in
V_0(\calC)$, then $y(P)$ divides $P_\calC(a)$.  Note, $P_\calC$ does
{\em not} satisfy the conjugation or duality rules.

Let $X : C_\calD\to C_\calD$ be given by the formula:
\[ X(a) = \begin{cases}
  x(P)a & \text{$x(P)$ does not divide $a$} \\
  0 & \text{otherwise,}
\end{cases} \]
where $a\in V(I)$ and $x(P)$ denotes the circle of $I$ which meets the
point $P$.  Set $\bP^0(\bt) = X$.

In what follows, we list only the non-zero terms of $P_\calC$.  Now we
define $P_\calC$ for $1$-dimensional configurations. We make the
following definitions:
\begin{definition}
  Let $\calC = (x, \gamma)$ be a $1$-dimensional configuration where
  $\gamma$ is a split arc.  Then $\calC$ has one active starting
  circle $x$ and two active ending circles.  Let $y_1$ denote the
  active ending circle which meets the tail of $\gamma^*$ and $y_2$
  the active ending circle which meets the head of $\gamma^*$.  If
  $y_2 = y(P)$, we define
  \[ P_{\calC_0}(1) = y_2. \]
\end{definition}
\begin{definition}
  Let $\calC = (x_1, x_2, \gamma)$ be a $1$-dimensional configuration
  where $\gamma$ is a join arc.  Then $\calC$ has two active starting
  circles and one active ending circle $y$.  Let $x_1$ denote the
  active starting circle that meets the tail of $\gamma$ and $x_2$ the
  active starting circle that meets the head of $\gamma$.  If $x_1 =
  x(P)$, we define
  \[ P_{\calC_0}(x_2) = y. \]
\end{definition}

\begin{figure}
  \begin{center}
    \includegraphics{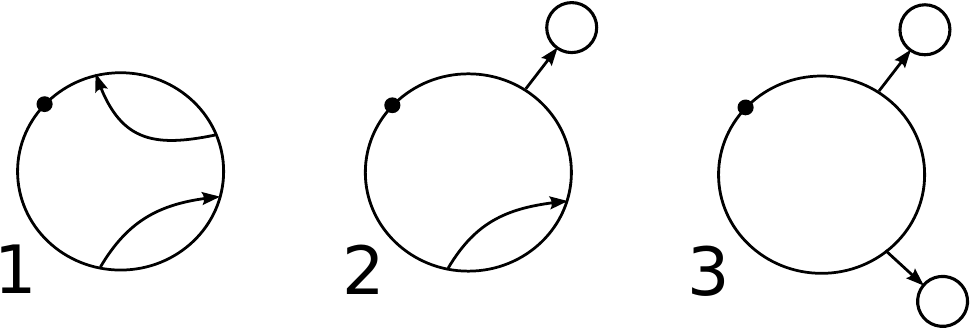}
  \end{center}
  \caption{The three $2$-dimensional configurations up to naturality
    and rotation for which $P_\calC\neq 0$.\label{p2dconfigs}}
\end{figure}

Now we define $P_\calC$ for $2$-dimensional configurations.  Up to
naturality and rotation, there are three $2$-dimensional
configurations for which $P_\calC\neq 0$.  They are given in
Figure~\ref{p2dconfigs}.  We make the following definitions:
\begin{definition}
  \begin{itemize}
  \item For a configuration of type 1, we define
    \[ P_{\calC_0}(1) = y(P). \]

  \item For a configuration of type 2, there are two starting circles.
    Let $x_1 = x(P)$ and $x_2$ be the other starting circle.  We
    define
    \[ P_{\calC_0}(x_1) = y(P). \]
    
  \item For a configuration of type 3, there are three starting
    circles.  Let $x_1 = x(P)$ and let $x_2$ and $x_3$ denote the
    other starting circles which meet a single arc.  We define
    \[ P_{\calC_0}(x_2x_3) = y(P). \]
  \end{itemize}
\end{definition}

Finally, for $k > 2$, we make the following definition:
\begin{definition}
  A $k$-dimensional configuration $\calC = (x_1, \dotsc, x_s,
  \gamma_1, \dotsc, \gamma_k)$ with $p + 1$ active starting circles
  and $q + 1$ active ending circles is said to be of type $P_{p,q}$ if,
  for each pair $(i, j)$ with $1\le i < j\le k$, the $2$-dimensional
  configuration $(x_1, \dotsc, x_s, \gamma_i, \gamma_j)$ is of type 1,
  2 or 3.  There is a unique starting circle $x_1 = x(P)$ called the
  central starting circle.  The other active starting circles $x_i$
  meet a single arc.  Similarly there is a unique ending circle $y_1 =
  y(P)$ that meets all the all the dual arcs $\gamma_i^*$.  In this
  case, we define
  \[ P_{\calC_0}(x_2x_3\dotsm x_{p+1}) = y_1 \]
  when $p\ge 1$ and
  \[ P_{C_0}(1) = y_1 \]
  when $p = 0$.
\end{definition}

We aim to show $\bP(\bt)$ is a chain map.

Recall the definition of the edge homotopy maps $H_m$ from
\cite{szaboss}.  The map $H_\calC$ which satisfies the extension
property.  The map $H_\calC$ is nonzero only for $1$-dimensional
configurations.  There are only two kinds of purely active
$1$-dimensional configurations: split and join.  For a $1$-dimensional
split configuration $\calC$, $H_\calC$ is given by
\[ H_\calC(1) = 1. \]
For a $1$-dimensional join configuration $\calC$ with starting circles
$x_1$ and $x_2$ and ending circle $y$, $H_\calC$ is given by
\[ H_\calC(x_1x_2) = y. \]
Note that $H_\calC$, like the Khovanov differential, does not depend
on the orientation of the arcs.  Finally, we define
\[ H_m = \sum_{(I, J)} H_{I, J} \]
the sum is taken over $1$-dimensional faces $(I, J)$ which differ only
$m^\text{th}$ coordinate and where
\[ H_{I,J} = H_{\calC(I, J)}. \]

\begin{lemma}\label{deformlem}
  Suppose $\bt$ and $\bt'$ are decorations of the diagram $\calD$ that
  differ only at the $m^\text{th}$ crossing.  Then $P(\bt)$ and
  $P(\bt')$ are related by the following formula:
  \[ P(\bt') = P(\bt) + H_m P(\bt) + P(\bt) H_m. \]
\end{lemma}
\begin{proof}
  The proof is analogous to the proof of Theorem 5.4 in
  \cite{szaboss}.
  
  Let $\delta$ denote the (unoriented) arc corresponding to the
  $m^\text{th}$ crossings.  It is sufficient two show the following
  equation holds:
  \begin{equation}\label{deformeq}
    P_{I,J,\bt} - P_{I,J,\bt} = P_{I',J}H_{I,I'} + H_{J',J}P_{I,J'},
  \end{equation}
  for all $k$-faces $(I, J)$ with $I(m) = 0$ and $J(m) = 1$ where $I'$
  is obtained from $I$ by changing the $m^\text{th}$ coordinate to $1$
  and $J'$ is obtained from $J$ by changing the $m^\text{th}$
  coordinate to $0$ and the decoration is omitted from the notation on
  the right hand side since $\bt$ and $\bt'$ agree on $(I, J')$ and
  $(I', J)$.
  
  If $\undecC = \undecC(I, J)$ is disconnected, then the left hand
  side of \ref{deformeq} is zero.  Then either both $\calC(I,J')$ and
  $\calC(I',J)$ are disconnected or $\calC(I,I')$ and $\calC(I',J)$
  are disjoint and $\calC(I,J')$ and $\calC(J,J')$ are disjoint.  In
  both cases, the right hand side is zero; the latter follows from the
  extension property.
  
  We assume $\undecC(I,J)$ is connected.  The case when $\undecC$ is
  $1$- or $2$-dimensional is left as a straightforward exercise for
  the reader.  We now assume $k\ge 3$.
  
\begin{figure}
  \begin{center}
    \includegraphics{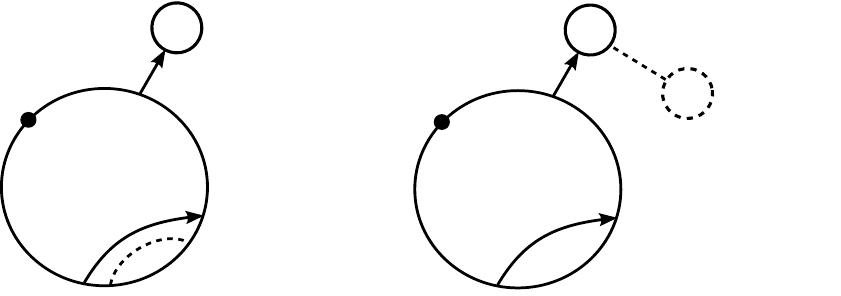}
  \end{center}
  \caption{Possible positions of the arc $\delta$ in proof of
    Lemma~\ref{deformlem} when $H_{J',J}P_{I,J'}\neq
    0$.\label{deformfig}}
\end{figure}
  
\begin{figure}
  \begin{center}
    \includegraphics{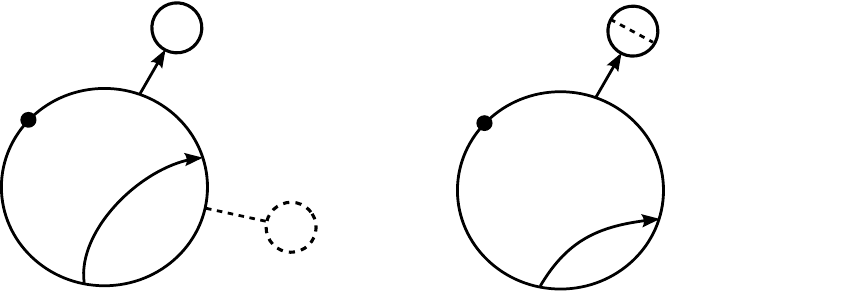}
  \end{center}
  \caption{Possible positions of the dual arc $\delta^*$ in proof of
    Lemma~\ref{deformlem} when $P_{I',J}H_{I,I'} \neq
    0$.\label{deformfig2}}
\end{figure}

  First, suppose $P_{I,J,\bt}$ or $P_{I,J,\bt'}$ is non-zero; we can
  assume it is $P_{I,J,\bt}$.  Then $P_{I,J,\bt'} = 0$.  If the arc
  $\delta$ is join, then $P_{I,J,\bt} = H_{J',J}P_{I,J'}$ and
  $P_{I',J}H_{I,I'} = 0$.  Alternatively, if $\delta$ is a split, then
  $P_{I,J,\bt} = P_{I',J}H_{I,I'}$ and $H_{J',J}P_{I,J'} = 0$.
  
  Now, suppose
  \begin{equation}\label{eqlhszero}
    P_{I,J,\bt} = P_{I,J,\bt'} = 0.
  \end{equation}
  In what follows, we need to show that
  \begin{equation}\label{eqrhszero}
    P_{I',J}H_{I,I'} = H_{J',J}P_{I,J'}
  \end{equation}
  
  First, suppose $H_{J',J}P_{I,J'}\neq 0$.  If $\delta$ is a split
  arc, it must split one of circles $y_2, \dotsc, y_m$ and
  \ref{eqrhszero} holds.  Alternatively, if $\delta$ is a join arc, it
  must connect a new circle $w$ to $y_1$.  The arc $\delta$ cannot
  meet $w$ and $x_1$, since that would contradict \ref{eqlhszero}.
  Therefore, $\delta$ meets $w$ and $x_i$ for some $2\le i\le p + 1$
  and again \ref{eqrhszero} holds.  See Figure~\ref{deformfig}.
  
  Alternatively, suppose $P_{I',J}H_{I,I'} \neq 0$.  If $\delta$ is a
  split arc, it must split a new circle $w$ off the central circle
  $x_1$.  The arc $\delta$ cannot split a circle off a portion of
  $x_1$ that lies in $y_1$, since that would contradict the fact
  \ref{eqlhszero}.  Therefore, $\delta$ splits a circle off a portion
  of $x_1$ that lies in $y_i$ for $i\le 2\le q+1$ and \ref{eqrhszero}
  holds.  If $\delta$ is a join, it must join a new circle $w$ to a
  circle $x_i$ for $2\le i\le p+1$ and again \ref{eqrhszero} holds.
  See Figure~\ref{deformfig2}.
\end{proof}

\begin{lemma}
  Let $\bt$ and $\bt'$ be two decorations of the diagram $\calD$.
  Then $P(\bt)$ is a chain map with respect to $\bd(\bt)$ if and only
  if $P(\bt')$ is a chain map with respect to $\bd(\bt')$.
\end{lemma}
\begin{proof}
  It is enough to consider the case when $\bt$ and $\bt'$ differ at a
  single crossing, say the $m^\text{th}$.  The claim follows by
  Lemma~\ref{deformlem}, Theorem 5.4 in \cite{szaboss} and the fact
  that $H_m$ is zero on $V(I)$ when the $m^\text{th}$ coordinate of
  $I$ is $1$.
\end{proof}

\begin{lemma}\label{decoratelem}
  Let $(I, J)$ be an undecorated $k$-face.  There exists a decoration
  $\bt$ of the diagram $\calD$ such that
  \begin{equation}\label{eq1}
    \sum_K d_{K,J,\bt} P_{I,K,\bt} = \sum_K P_{K,J,\bt} d_{I,J,\bt},
  \end{equation}
  where then sums are taken over resolutions $K$ such that $I < K <
  J$.
\end{lemma}
\begin{proof}
  Set $\undecC = \undecC(I, J) = (x_1, \dotsc, x_t, \gamma_1, \dotsc,
  \gamma_k)$.  By the extension property it suffices to consider only
  active configurations $\undecC$.
  
  First, we consider the case when $\undecC$ is disconnected.  Let
  $\bt$ be arbitrary.  If $\undecC$ has more than 2 components, then
  at least one of $\calC(I, K, \bt)$ or $\calC(I, K, \bt)$ is
  disconnected.  By the disconnected rule, both sides of \eqref{eq1}
  vanish.  If $\undecC$ has 2 components, the only non-zero terms in
  \eqref{eq1} will be when $\calC(I, K, \bt)$ is one component and
  $\calC(K, J, \bt)$ is the other.  In this case, the claim holds by
  the extension property.
  
  If none of the circles of $\undecC$ go through the point $P$, then
  both sides of \eqref{eq1} necessarily vanish (after choice of some
  $\bt$).
  
  We assume $\undecC$ is connected and contains the point $P$.  By
  orienting the arcs that meet $x(P)$ appropriately, we can choose a
  decoration $\bt$ such that $P_{I,K,\bt} = 0$ for all $K$.  Then the
  left hand side of \eqref{eq1} vanishes.  Our goal is to choose the
  orientation of the remaining arcs appropriately to make the right
  hand side vanish as well.
  
  Let $x_1 = x(P)$ in $\calC = \calC(I,J,\bt)$.  Suppose there are two
  circles $x_i, i = 1, 2$ with the property that there are two (or
  more) arcs which meet $x_i$ and $x_1$.  Since $P_{\calC}$ vanishes
  on configurations where multiple arcs meeting the same pair of
  circles, the only non-zero terms on the right hand side can arise
  from a resolution $K$ such that $\calC_1 = \calC(I,K,\bt)$ has type
  $E$ and includes an arc meeting $x_1$ and $x_2$ and an arc meeting
  $x_1$ and $x_2$.  Set $\calC_2 = \calC(K,J,\bt)$.  The central
  circle of $\calC_1$ is marked, so $P_{K,J,\bt}$ vanishes unless $P$
  meets a circle $y_i$ for $2\le i\le q+1$.  However, the orientation
  of $\bt$ the right hand side is zero in this case.  Thus, we can
  assume there is at most one such circle $x_2$.  Consider the
  resolution $K$ where the arcs meeting $x_1$ and $x_2$ have
  $1$-resolution.  Let $x' = x(P)$ denote the circle meeting $P$ in
  $K$.  Orient the remaining arcs (that is, those that did not meet
  $x_1$) so that $P_{K,K',\bt} = 0$ for all $K < K'$.  We claim $\bt$
  is our desired decoration.
  
  Let $\gamma_i$ be an arc that meets $x(P)$.  Surgery along arcs that
  do not meet $x(P)$ cannot change the type of the $1$-configuration
  $(x_1, \dotsc, x_t, \gamma_i)$.  Consider the left hand summand for
  some resolution $K$.  Set $\calC_1 = \calC(I,K,\bt)$ and $\calC_2 =
  \calC(K,J,\bt)$.  If $x(P)$ is not among the active circles of
  $\calC_1$, then the summand vanishes.  It remains to consider $K$
  for which $\calC_1$ is one of the five configuration types.
  Moreover, $P_{K,J,\bt}$ vanishes on monomials divisible by $x(P)$,
  so we need only consider $K$ when $y(P)$ does not divide the image
  of $d_{I,K,\bt}$.  There are three cases: when $\calC_1$ has type
  $A, C$ or $E$.

\begin{figure}
  \begin{center}
    \includegraphics{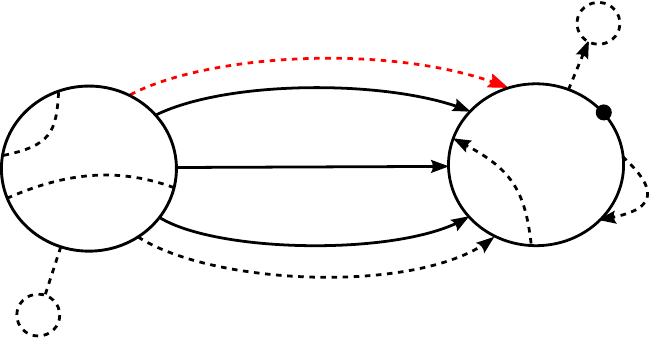}
  \end{center}
  \caption{The case when $\calC_1$ has type $A_k$ in proof of
    Lemma~\ref{decoratelem}.\label{casea}}
\end{figure}

\begin{figure}
  \begin{center}
    \includegraphics{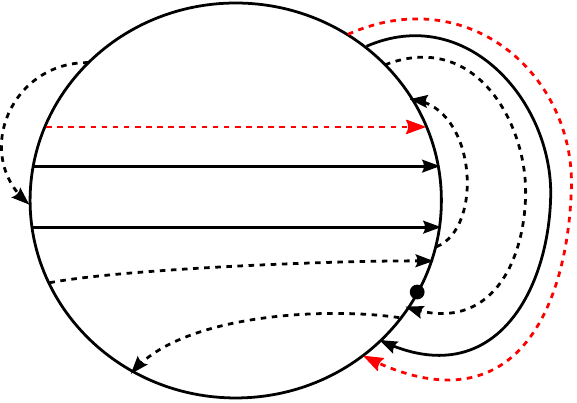}
  \end{center}
  \caption{The case when $\calC_1$ has type $C_{p,q}$ in proof of
    Lemma~\ref{decoratelem}.\label{casec}}
\end{figure}

  First, consider the case when $\calC_1$ has type $A_k$ as shown in
  Figure~\ref{casea}.  There are two active starting circles, $x_1 =
  x(P)$ which meets the head of each arc of $\calC_1$ and $x_2$ which
  meets the tail of each arc of $\calC_1$.  There are $k$ active
  ending circles.  Set $y_1 = y(P)$ and let $y_2, \dotsm, y_k$ denote
  the other active ending circles.  The orientation of arcs that do
  not meet $x_1$ have been chosen so that $P_{K,J,\bt} = 0$.
  
  Now consider the case there are no arcs of $\calC_2$ that fail to
  meet $x_1$.  Join arcs and arcs that meet a passive circle of
  $\calC_1$ already have orientation chosen so that $P_{K,J,\bt} = 0$.
  Suppose an arc meets one of the ending circles $y_i$ for $2\le i\le
  k$.  Then again we have $P_{K,J,\bt} = 0$ since the image of
  $d_{I,K,\bt}$ is not divisible by $y_i$.  That leaves one remaining
  case: $\calC_2$ has a single arc the head of which meets $x_1$ and
  the tail of which meets $x_2$.  Such an arc is shown in red in
  Figure~\ref{casea}.  Then $\calC(I,J,\bt)$ has type $A$ and explicit
  computation shows the right hand side is zero.
  
  Next, suppose $\calC_1$ has type $C_{p,q}$ as shown in
  Figure~\ref{casec}.  Again, case analysis again shows the that
  $P_{K,J,\bt} = 0$ except when there are (possibly) arcs of the form
  shown in red in the figure.  In either case, $\calC(I,J,\bt)$ has
  type $C$ and explicit computation shows the right hand side is zero.
  
  Finally, suppose $\calC_1$ has type $E_{p,q}$.  The analysis above
  showed the the point $P$ must be on an edge that meets a circle
  $y_i$ for $2\le i\le q+1$, there is only one split arc in $\calC_1$
  and there are no join arcs.  Again by case analysis, we see
  $P_{K,J,\bt} = 0$ except for a single edge of the type shown in red.
  The configuration $\calC(I,J,\bt)$ is $2$-dimensional and it is
  easily seen that the right hand side vanishes.
\end{proof}

We our now ready to complete the proof of
Proposition~\ref{babytwinarrows}.  The image of $\bP(\bt)$ lies in the
subcomplex $\rdC(\calD, \bt, P)$ and $\bP(\bt)$ vanishes on on
$\rdC(\calD, \bt, P)$.  Therefore, $P$ descends to a map $\qbP(\bt) :
\qC(\calD, \bt, P)\to \rdC(\calD, \bt, P)$.  The map $X$ induces an
isomorphism between $\qKh(L)$ and $\rdKh(L)$; see \cite{shumakovitch}.
Thus, $\qbP(\bt)$ induces an isomorphism between $\qE^2(L, c)$ and
$\rdE^2(L, c)$.  By the mapping lemma for spectral sequences
$\qbP(\bt)$ induces an isomorphism $\qE^k(L, c)\cong \rdE(L, c)$ for
$k\ge 2$.
\end{proof}

Let $\calD$ be a diagram with a distinguished crossing $c$.  Let
$\calD_0$ and $\calD_1$ be the diagrams obtained by replacing $c$ by
its $0$ or $1$ resolution, respectively.  Recall \cite{oszdouble} the
set of quasi-alternating links $\calQ$ is the smallest set of links
satisfying
\begin{itemize}
\item the unknot is in $\calQ$, and
\item if $L$ is a link which admits a diagram $\calD$ with a
  distinguished crossing $c$ such that $\calD_0, \calD_1\in \calQ$ and
  $\det(\calD_0), \det(\calD_1)\neq 0$ and $\det(L) = \det(\calD_0) +
  \det(\calD_1)$, then $L\in \calQ$.
\end{itemize}
All alternating links are quasi-alternating.  Ozsv\'ath and Szab\'o
\cite{oszdouble} showed that the Heegaard Floer variant of the
spectral sequence $E_{HF}^k$ collapses at the $E_{HF}^2$-page for
quasi-alternating links.  A link $L$ is called $\delta$-thin if its
reduced Khovanov homology is supported in a single $\delta$-grading.
Manolescu and Ozsv\'ath \cite{mos} showed that quasi-alternating knots
are $\delta$-thin.  Since the higher differentials $\bd^k(\bt)$
decrease the $\delta$-grading by $2$, the reduced spectral sequence
$\rdE^k$ necessarily collapses at the $\rdE^2$-page for $\delta$-thin
links.  Thus, we have shown the following proposition:
\begin{proposition}
  Let $L$ be a $\delta$-thin knot.  Then the reduced spectral sequence
  $\rdE^k(L)$ collapses at the $\rdE^2$-page.
\end{proposition}
Combined with the twin arrows conjecture, this would imply the
spectral sequence $E^k(L)$ collapses at the $E^2$-page for
$\delta$-thin links.

Using the transverse invariant, Baldwin \cite{baldwin} computed
$E_{HF}^k$ for the torus knots $T(3,4)$ and $T(3,5)$ with a
well-defined quantum grading.  Our computations agree with his.  In
addition, he showed that $E_{HF}^k$ distributes over connect sum of
links.  We prove the following analogous result.
\begin{proposition}
  Let $L$ and $L'$ be links with distinguished components $c, c'$,
  respectively.  Let $L\# L'$ denote the connect sum between the
  distinguished components and let $c''$ denote the resulting
  component.  Then
  \[ \rdE^k(L\# L', c'') \cong \rdE^k(L, c)\otimes \rdE^k(L', c')\{1\} \]
  for $k\ge 2$.
\end{proposition}
\begin{proof}
  We show the underlying chain complexes are equal.  Choose decorated
  diagrams diagrams $(\calD, \bt)$ and $(\calD', \bt')$ of $L, L'$,
  respectively, such that a decorated diagram $(\calD'', \bt'')$ of
  $L\# L'$ is obtained from the disjoint union of $\calD$ and $\calD'$
  by surgery along an arc connecting the distinguished components.
  Choose points $P$ and $P'$ on the edges of $\calD, \calD'$,
  respectively, meeting the connect sum surgery arc.  Choose a point
  $P''$ of $\calD''$ on either edge resulting from the surgery; they
  will always belong to the same circle.  We have $C_{\calD}\otimes
  C_{\calD'}\{1\}$ is isomorphic to $C_{\calD''}$ as graded vector
  spaces, the grading shift coming from the fact the connect sum joins
  two circles which always carry the generator with grading $-1$.  The
  tensor product accounts for all faces which resolve only crossings
  of $\calD$ or $\calD'$.  We argue $\bd''(\bt)$ has no other terms.
  Inspection shows that only configurations of type $E$ can be
  expressed as a connect sum of two nontrivial configurations.
  However, such a decomposition will be along the central circle of
  the configuration which in this case is marked.  Thus, such
  configurations do not contribute to $\bd''(\bt)$.
\end{proof}

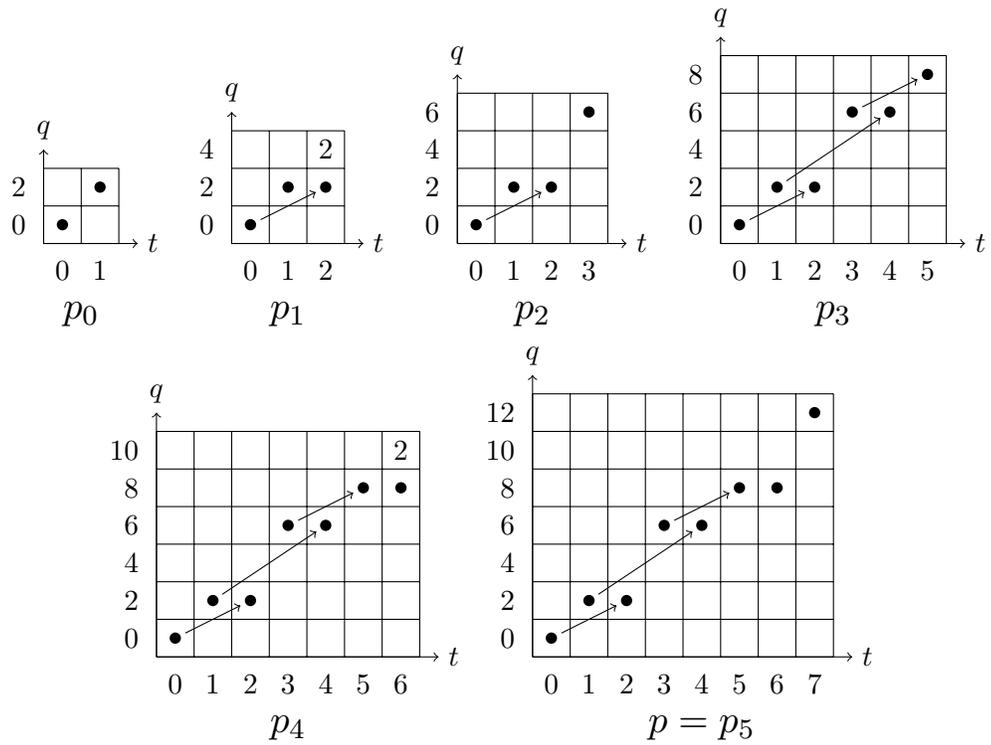
\begin{figure}
\begin{tikzpicture}[scale=.5]
  \begin{scope}
    \pgftransformshift{\pgfpointxy{0}{11}}
    \draw[->] (0,0) -- (2.5,0) node[right] {$t$};
    \draw[->] (0,0) -- (0,2.5) node[above] {$q$};
    \draw[step=1] (0,0) grid (2,2);
    \foreach \t in {0, 1}
      \draw (\t + 0.5,-.2) node[below] {$\t$};
    \foreach \t / \s in {0/0, 1/2}
      \draw (-.2, \t + 0.5) node[left] {$\s$};
    \draw(1,-1.2) node[below,scale=1.33] {$p_0$};
    \fill (0.5, 0.5) circle (.15);
    \fill (1.5, 1.5) circle (.15);
  \end{scope}
  \begin{scope}
    \pgftransformshift{\pgfpointxy{5}{11}}
    \draw[->] (0,0) -- (3.5,0) node[right] {$t$};
    \draw[->] (0,0) -- (0,3.5) node[above] {$q$};
    \draw[step=1] (0,0) grid (3,3);
    \foreach \t in {0, 1, 2}
      \draw (\t + 0.5,-.2) node[below] {$\t$};
    \foreach \t / \s in {0/0, 1/2, 2/4}
      \draw (-.2, \t + 0.5) node[left] {$\s$};
    \draw(1.5,-1.2) node[below,scale=1.33] {$p_1$};
    \fill (0.5, 0.5) circle (.15);
    \fill (1.5, 1.5) circle (.15);
    \fill (2.5, 1.5) circle (.15);
    \fill (2.5, 2.5) node {$2$};
    \draw[->] (0.768, 0.634) -- (2.232, 1.366);
  \end{scope}
  \begin{scope}
    \pgftransformshift{\pgfpointxy{11}{11}}
    \draw[->] (0,0) -- (4.5,0) node[right] {$t$};
    \draw[->] (0,0) -- (0,4.5) node[above] {$q$};
    \draw[step=1] (0,0) grid (4,4);
    \foreach \t in {0, 1, 2, 3}
      \draw (\t + 0.5,-.2) node[below] {$\t$};
    \foreach \t / \s in {0/0, 1/2, 2/4, 3/6}
      \draw (-.2, \t + 0.5) node[left] {$\s$};
    \draw(2,-1.2) node[below,scale=1.33] {$p_2$};
    \fill (0.5, 0.5) circle (.15);
    \fill (1.5, 1.5) circle (.15);
    \fill (2.5, 1.5) circle (.15);
    \fill (3.5, 3.5) circle (.15);
    \draw[->] (0.768, 0.634) -- (2.232, 1.366);
  \end{scope}
  \begin{scope}
    \pgftransformshift{\pgfpointxy{18}{11}}
    \draw[->] (0,0) -- (6.5,0) node[right] {$t$};
    \draw[->] (0,0) -- (0,5.5) node[above] {$q$};
    \draw[step=1] (0,0) grid (6,5);
    \foreach \t in {0, 1, 2, 3, 4, 5}
      \draw (\t + 0.5,-.2) node[below] {$\t$};
    \foreach \t / \s in {0/0, 1/2, 2/4, 3/6, 4/8}
      \draw (-.2, \t + 0.5) node[left] {$\s$};
    \draw(3,-1.2) node[below,scale=1.33] {$p_3$};
    \fill (0.5, 0.5) circle (.15);
    \fill (1.5, 1.5) circle (.15);
    \fill (2.5, 1.5) circle (.15);
    \fill (3.5, 3.5) circle (.15);
    \fill (4.5, 3.5) circle (.15);
    \fill (5.5, 4.5) circle (.15);
    \draw[->] (0.768, 0.634) -- (2.232, 1.366);
    \draw[->] (1.750, 1.666) -- (4.250, 3.334);
    \draw[->] (3.768, 3.634) -- (5.232, 4.366);
  \end{scope}
  \begin{scope}
    \pgftransformshift{\pgfpointxy{3}{0}}
    \draw[->] (0,0) -- (7.5,0) node[right] {$t$};
    \draw[->] (0,0) -- (0,6.5) node[above] {$q$};
    \draw[step=1] (0,0) grid (7,6);
    \foreach \t in {0, 1, 2, 3, 4, 5, 6}
      \draw (\t + 0.5,-.2) node[below] {$\t$};
    \foreach \t / \s in {0/0, 1/2, 2/4, 3/6, 4/8, 5/10}
      \draw (-.2, \t + 0.5) node[left] {$\s$};
    \draw(3.5,-1.2) node[below,scale=1.33] {$p_4$};
    \fill (0.5, 0.5) circle (.15);
    \fill (1.5, 1.5) circle (.15);
    \fill (2.5, 1.5) circle (.15);
    \fill (3.5, 3.5) circle (.15);
    \fill (4.5, 3.5) circle (.15);
    \fill (5.5, 4.5) circle (.15);
    \fill (6.5, 4.5) circle (.15);
    \fill (6.5, 5.5) node {$2$};
    \draw[->] (0.768, 0.634) -- (2.232, 1.366);
    \draw[->] (1.750, 1.666) -- (4.250, 3.334);
    \draw[->] (3.768, 3.634) -- (5.232, 4.366);
  \end{scope}
  \begin{scope}
    \pgftransformshift{\pgfpointxy{13}{0}}
    \draw[->] (0,0) -- (8.5,0) node[right] {$t$};
    \draw[->] (0,0) -- (0,7.5) node[above] {$q$};
    \draw[step=1] (0,0) grid (8,7);
    \foreach \t in {0, 1, 2, 3, 4, 5, 6, 7}
      \draw (\t + 0.5,-.2) node[below] {$\t$};
    \foreach \t / \s in {0/0, 1/2, 2/4, 3/6, 4/8, 5/10, 6/12}
      \draw (-.2, \t + 0.5) node[left] {$\s$};
    \draw(4.5,-1.2) node[below,scale=1.33] {$p = p_5$};
    \fill (0.5, 0.5) circle (.15);
    \fill (1.5, 1.5) circle (.15);
    \fill (2.5, 1.5) circle (.15);
    \fill (3.5, 3.5) circle (.15);
    \fill (4.5, 3.5) circle (.15);
    \fill (5.5, 4.5) circle (.15);
    \fill (6.5, 4.5) circle (.15);
    \fill (7.5, 6.5) circle (.15);
    \draw[->] (0.768, 0.634) -- (2.232, 1.366);
    \draw[->] (1.750, 1.666) -- (4.250, 3.334);
    \draw[->] (3.768, 3.634) -- (5.232, 4.366);
  \end{scope}
\end{tikzpicture}
  \caption{Diagrammatic representation of the Poincar\'e polynomials
    $p^k_\ell(t, q)$. \label{pfig}}
\end{figure}

Bloom \cite{bloomss} conjectured the structure of the monopole Floer
homology variant of the spectral sequence for torus knots $T(3, 6n\pm
1)$.  Our computations agree with his conjecture for $T(3,5)$ and
$T(3,7)$.  Based on our computations for torus links $T(3,n)$ with
$n\le 9$, we extend his conjecture to all $3$-strand torus links as
follows.
\begin{conjecture}
  Set
  \begin{align*}
    f_j(t, q) &= \sum_{i=0}^{j-1} t^{8i}q^{12i} \\
    p^2_0 (t, q) = p^3_0(t, q) = p^4_0(t, q) &= 1 + tq^2 \\
    p^2_1(t, q) &= 1 + tq^2 + t^2q^2 + 2t^2q^4 \\
    p^3_1(t, q) = p^4_1(t, q) &= tq^2 + 2t^2q^4 \\
    p^2_2(t, q) &= 1 + tq^2 + t^2q^2 + t^3q^6 \\
    p^3_2(t, q) = p^4_2(t, q) &= tq^2 + t^3q^6 \\
    p^2_3(t, q) &= 1 + tq^2 + t^2q^2 + t^3q^6 + t^4q^6 + t^5q^8 \\
    p^3_3(t, q) &= tq^2 + t^4q^6 \\
    p^4_3(t, q) &= 0 \\
    p^2_4(t, q) &= 1 + tq^2 + t^2q^2 + t^3q^6 + t^4q^6 + t^5q^8 + t^6q^8 + 2t^6q^{10} \\
    p^3_4(t, q) &= tq^2 + t^4q^6 + t^6q^8 + 2t^6q^{10} \\
    p^4_4(t, q) &=  t^6q^8 + 2t^6q^{10} \\
    p^2(t, q) = p^2_5(t, q) &= 1 + tq^2 + t^2q^2 + t^3q^6 + t^4q^6 + t^5q^8 + t^6q^8 + t^7q^{12} \\
    p^3(t, q) = p^3_5(t, q) &= tq^2 + t^4q^6 + t^6q^8 + t^7q^{12} \\
    p^4(t, q) = p^4_5(t, q) &= t^6q^8 + t^7q^{12}. \\
  \end{align*}
  For all $n > 1$, the spectral sequence for the torus link $T(3, n)$
  collapses at the $E^4$-page.  Moreover, writing $n = 2 + 6j + \ell$
  with $j\ge 0$ and $0\le \ell < 6$, the Poincar\'e polynomial for
  $\rdE^k(T(3, n)), k = 2, 3, 4$ is given by
  \[ \rdP^k_{3,n}(t, q) = q^{2j - 3}(1 + t^2q^4(f_j(t, q)p^k(q, t) + t^{7j}q^{12j}p^k_\ell(q, t))). \]
\end{conjecture}
Diagrammatic representations of the Poincar\'e polynomials $p^k_\ell(t,
q)$ are given in Figure~\ref{pfig}.

Bloom \cite{bloomodd} showed that odd Khovanov homology is mutation
invariant.  In addition, his argument showed that both the Heegaard
Floer and monopole Floer variants of the spectral sequence are
mutation invariant.  Using the mutant knot tables compiled by Stoimeno
\cite{stoimenow}, we verified that $E^k$ is invariant under mutation
for all (5300) mutant knot groups through $14$ crossings.  Thus we
make the following conjecture.
\begin{conjecture}
  Let $K$ be a knot.  For $k\ge 2$, $E^k(K)$ is invariant under
  mutation.
\end{conjecture}

Recall, $E'(L)$ denotes the spectral sequence arising from the
alternate differential defined using mirror configuration types.  We
make the following conjecture:
\begin{conjecture}
  Let $L$ be a knot or link.  For $k\ge 2$, we have $E^k(L) \cong
  E'^k(L)$.
\end{conjecture}

\appendix
\section{Results}

\begin{scriptsize}
\[ \begin{array}{r|ccl}
  \text{Knot $K$} & \rdE^k & \rank \rdE^k & \rdP^k(q, t) = \sum_{i,j} (\rank \rdE^k_{i,j})t^iq^j \\
  \hline
  8_{19} & \rdE^{2} & 5 & q^{5}+t^{2}q^{9}+t^{3}q^{11}+t^{4}q^{11}+t^{5}q^{15} \\
  & \rdE^{3} & 3 & q^{5}+t^{3}q^{11}+t^{5}q^{15} \\
  9_{42} & \rdE^{2} & 9 & t^{-4}q^{-7}+t^{-3}q^{-5}+t^{-2}q^{-3}+2t^{-1}q^{-1}+q^{-1}+q+tq^{3}+t^{2}q^{5} \\
  & \rdE^{3} & 7 & t^{-4}q^{-7}+t^{-3}q^{-5}+2t^{-1}q^{-1}+q+tq^{3}+t^{2}q^{5} \\
  10_{124} & \rdE^{2} & 7 & q^{7}+t^{2}q^{11}+t^{3}q^{13}+t^{4}q^{13}+t^{5}q^{17}+t^{6}q^{17}+t^{7}q^{19} \\
  & \rdE^{3} & 3 & q^{7}+t^{3}q^{13}+t^{6}q^{17} \\
  & \rdE^{4} & 1 & q^{7} \\
  10_{128} & \rdE^{2} & 13 & q^{5}+tq^{7}+2t^{2}q^{9}+2t^{3}q^{11}+t^{4}q^{11}+2t^{4}q^{13}+2t^{5}q^{15}+t^{6}q^{17} \\
 & & & +t^{7}q^{19} \\
  & \rdE^{3} & 11 & q^{5}+tq^{7}+t^{2}q^{9}+2t^{3}q^{11}+2t^{4}q^{13}+2t^{5}q^{15}+t^{6}q^{17}+t^{7}q^{19} \\
  10_{132} & \rdE^{2} & 11 & t^{-7}q^{-15}+t^{-6}q^{-13}+t^{-5}q^{-11}+2t^{-4}q^{-9}+t^{-3}q^{-9}+t^{-3}q^{-7} \\
 & & & +t^{-2}q^{-7}+t^{-2}q^{-5}+t^{-1}q^{-3}+q^{-3} \\
  & \rdE^{3} & 5 & t^{-7}q^{-15}+t^{-6}q^{-13}+t^{-4}q^{-9}+t^{-3}q^{-7}+t^{-1}q^{-3} \\
  10_{136} & \rdE^{2} & 17 & t^{-3}q^{-9}+2t^{-2}q^{-7}+2t^{-1}q^{-5}+3q^{-3}+q^{-1}+3tq^{-1}+2t^{2}q \\
 & & & +2t^{3}q^{3}+t^{4}q^{5} \\
  & \rdE^{3} & 15 & t^{-3}q^{-9}+2t^{-2}q^{-7}+2t^{-1}q^{-5}+3q^{-3}+3tq^{-1}+t^{2}q+2t^{3}q^{3}+t^{4}q^{5} \\
  10_{139} & \rdE^{2} & 11 & q^{7}+t^{2}q^{11}+t^{3}q^{13}+t^{4}q^{13}+t^{5}q^{15}+t^{5}q^{17}+2t^{6}q^{17}+t^{7}q^{19} \\
 & & & +t^{8}q^{21}+t^{9}q^{23} \\
  & \rdE^{3} & 7 & q^{7}+t^{3}q^{13}+t^{5}q^{15}+2t^{6}q^{17}+t^{8}q^{21}+t^{9}q^{23} \\
  & \rdE^{4} & 5 & q^{7}+t^{5}q^{15}+t^{6}q^{17}+t^{8}q^{21}+t^{9}q^{23} \\
  10_{145} & \rdE^{2} & 13 & t^{-9}q^{-21}+t^{-8}q^{-19}+t^{-7}q^{-17}+2t^{-6}q^{-15}+t^{-5}q^{-15}+t^{-5}q^{-13} \\
 & & & +t^{-4}q^{-13}+t^{-4}q^{-11}+t^{-3}q^{-11}+t^{-3}q^{-9}+t^{-2}q^{-9}+q^{-5} \\
  & \rdE^{3} & 7 & t^{-9}q^{-21}+t^{-8}q^{-19}+t^{-6}q^{-15}+t^{-5}q^{-13}+t^{-3}q^{-11}+t^{-3}q^{-9}+q^{-5} \\
  & \rdE^{4} & 5 & t^{-9}q^{-21}+t^{-8}q^{-19}+t^{-5}q^{-13}+t^{-3}q^{-9}+q^{-5} \\
  10_{152} & \rdE^{2} & 19 & q^{7}+t^{2}q^{11}+t^{3}q^{13}+2t^{4}q^{13}+2t^{5}q^{15}+t^{5}q^{17}+3t^{6}q^{17}+3t^{7}q^{19} \\
 & & & +2t^{8}q^{21}+2t^{9}q^{23}+t^{10}q^{25} \\
  & \rdE^{3} & 15 & q^{7}+t^{3}q^{13}+t^{4}q^{13}+2t^{5}q^{15}+3t^{6}q^{17}+2t^{7}q^{19}+2t^{8}q^{21} \\
 & & & +2t^{9}q^{23}+t^{10}q^{25} \\
  & \rdE^{4} & 13 & q^{7}+t^{4}q^{13}+2t^{5}q^{15}+2t^{6}q^{17}+2t^{7}q^{19}+2t^{8}q^{21}+2t^{9}q^{23}+t^{10}q^{25} \\
  10_{153} & \rdE^{2} & 17 & t^{-5}q^{-11}+t^{-4}q^{-9}+t^{-3}q^{-7}+2t^{-2}q^{-5}+t^{-1}q^{-5}+t^{-1}q^{-3} \\
 & & & +q^{-3}+2q^{-1}+tq^{-1}+tq+2t^{2}q+t^{3}q^{3}+t^{4}q^{5}+t^{5}q^{7} \\
  & \rdE^{3} & 9 & t^{-5}q^{-11}+t^{-4}q^{-9}+t^{-2}q^{-5}+t^{-1}q^{-3}+q^{-1}+tq^{-1}+t^{2}q \\
 & & & +t^{4}q^{5}+t^{5}q^{7} \\
  & \rdE^{4} & 5 & t^{-5}q^{-11}+t^{-4}q^{-9}+q^{-1}+t^{4}q^{5}+t^{5}q^{7} \\
  10_{154} & \rdE^{2} & 21 & q^{5}+t^{2}q^{9}+t^{3}q^{9}+t^{3}q^{11}+3t^{4}q^{11}+2t^{5}q^{13}+t^{5}q^{15}+3t^{6}q^{15} \\
 & & & +3t^{7}q^{17}+2t^{8}q^{19}+2t^{9}q^{21}+t^{10}q^{23} \\
  & \rdE^{3} & 17 & q^{5}+t^{3}q^{9}+t^{3}q^{11}+2t^{4}q^{11}+2t^{5}q^{13}+3t^{6}q^{15}+2t^{7}q^{17} \\
 & & & +2t^{8}q^{19}+2t^{9}q^{21}+t^{10}q^{23} \\
  & \rdE^{4} & 15 & q^{5}+t^{3}q^{9}+2t^{4}q^{11}+2t^{5}q^{13}+2t^{6}q^{15}+2t^{7}q^{17}+2t^{8}q^{19} \\
 & & & +2t^{9}q^{21}+t^{10}q^{23} \\
  10_{161} & \rdE^{2} & 13 & q^{5}+t^{2}q^{9}+t^{3}q^{9}+t^{3}q^{11}+2t^{4}q^{11}+t^{5}q^{13}+t^{5}q^{15}+2t^{6}q^{15} \\
 & & & +t^{7}q^{17}+t^{8}q^{19}+t^{9}q^{21} \\
  & \rdE^{3} & 9 & q^{5}+t^{3}q^{9}+t^{3}q^{11}+t^{4}q^{11}+t^{5}q^{13}+2t^{6}q^{15}+t^{8}q^{19}+t^{9}q^{21} \\
  & \rdE^{4} & 7 & q^{5}+t^{3}q^{9}+t^{4}q^{11}+t^{5}q^{13}+t^{6}q^{15}+t^{8}q^{19}+t^{9}q^{21} \\
\end{array} \]
\end{scriptsize}
\begin{scriptsize}
\[ \begin{array}{r|ccl}
  \text{Knot $K$} & \rdE^k & \rank \rdE^k & \rdP^k(q, t) = \sum_{i,j} (\rank \rdE^k_{i,j})t^iq^j \\
  \hline
  11n6 & \rdE^{2} & 33 & t^{-7}q^{-15}+2t^{-6}q^{-13}+3t^{-5}q^{-11}+4t^{-4}q^{-9}+4t^{-3}q^{-7}+t^{-2}q^{-7} \\
 & & & +4t^{-2}q^{-5}+t^{-1}q^{-5}+3t^{-1}q^{-3}+q^{-3}+3q^{-1}+2tq^{-1} \\
 & & & +tq+t^{2}q+t^{3}q^{3}+t^{4}q^{5} \\
  & \rdE^{3} & 25 & t^{-7}q^{-15}+2t^{-6}q^{-13}+3t^{-5}q^{-11}+3t^{-4}q^{-9}+3t^{-3}q^{-7}+4t^{-2}q^{-5} \\
 & & & +2t^{-1}q^{-3}+q^{-3}+2q^{-1}+tq^{-1}+tq+t^{3}q^{3}+t^{4}q^{5} \\
  & \rdE^{4} & 21 & t^{-7}q^{-15}+2t^{-6}q^{-13}+3t^{-5}q^{-11}+3t^{-4}q^{-9}+2t^{-3}q^{-7}+3t^{-2}q^{-5} \\
 & & & +2t^{-1}q^{-3}+2q^{-1}+tq+t^{3}q^{3}+t^{4}q^{5} \\
  11n9 & \rdE^{2} & 29 & t^{-9}q^{-23}+2t^{-8}q^{-21}+2t^{-7}q^{-19}+3t^{-6}q^{-17}+t^{-5}q^{-17}+3t^{-5}q^{-15} \\
 & & & +t^{-4}q^{-15}+3t^{-4}q^{-13}+2t^{-3}q^{-13}+2t^{-3}q^{-11}+2t^{-2}q^{-11} \\
 & & & +t^{-2}q^{-9}+2t^{-1}q^{-9}+2q^{-7}+tq^{-5}+t^{2}q^{-3} \\
  & \rdE^{3} & 17 & t^{-9}q^{-23}+2t^{-8}q^{-21}+t^{-7}q^{-19}+2t^{-6}q^{-17}+2t^{-5}q^{-15}+2t^{-4}q^{-13} \\
 & & & +t^{-3}q^{-13}+t^{-3}q^{-11}+t^{-2}q^{-11}+t^{-1}q^{-9}+q^{-7}+tq^{-5}+t^{2}q^{-3} \\
  & \rdE^{4} & 11 & t^{-9}q^{-23}+2t^{-8}q^{-21}+t^{-7}q^{-19}+t^{-6}q^{-17}+t^{-5}q^{-15}+t^{-4}q^{-13} \\
 & & & +t^{-3}q^{-11}+q^{-7}+tq^{-5}+t^{2}q^{-3} \\
  11n12 & \rdE^{2} & 19 & t^{-7}q^{-15}+2t^{-6}q^{-13}+2t^{-5}q^{-11}+3t^{-4}q^{-9}+t^{-3}q^{-9}+3t^{-3}q^{-7} \\
 & & & +t^{-2}q^{-7}+2t^{-2}q^{-5}+2t^{-1}q^{-3}+q^{-3}+q^{-1} \\
  & \rdE^{3} & 13 & t^{-7}q^{-15}+2t^{-6}q^{-13}+t^{-5}q^{-11}+2t^{-4}q^{-9}+3t^{-3}q^{-7}+t^{-2}q^{-5} \\
 & & & +2t^{-1}q^{-3}+q^{-1} \\
  11n19 & \rdE^{2} & 11 & t^{-3}q^{-9}+t^{-2}q^{-9}+t^{-2}q^{-7}+t^{-1}q^{-7}+q^{-5}+q^{-3}+2tq^{-3} \\
 & & & +t^{2}q^{-1}+t^{3}q+t^{4}q^{3} \\
  & \rdE^{3} & 7 & t^{-2}q^{-9}+t^{-2}q^{-7}+q^{-5}+2tq^{-3}+t^{3}q+t^{4}q^{3} \\
  & \rdE^{4} & 5 & t^{-2}q^{-9}+q^{-5}+tq^{-3}+t^{3}q+t^{4}q^{3} \\
  11n20 & \rdE^{2} & 25 & t^{-4}q^{-11}+2t^{-3}q^{-9}+3t^{-2}q^{-7}+4t^{-1}q^{-5}+4q^{-3}+q^{-1} \\
 & & & +4tq^{-1}+3t^{2}q+2t^{3}q^{3}+t^{4}q^{5} \\
  & \rdE^{3} & 23 & t^{-4}q^{-11}+2t^{-3}q^{-9}+3t^{-2}q^{-7}+4t^{-1}q^{-5}+4q^{-3}+4tq^{-1} \\
 & & & +2t^{2}q+2t^{3}q^{3}+t^{4}q^{5} \\
  11n24 & \rdE^{2} & 25 & t^{-3}q^{-9}+2t^{-2}q^{-7}+3t^{-1}q^{-5}+4q^{-3}+q^{-1}+4tq^{-1}+4t^{2}q \\
 & & & +3t^{3}q^{3}+2t^{4}q^{5}+t^{5}q^{7} \\
  & \rdE^{3} & 23 & t^{-3}q^{-9}+2t^{-2}q^{-7}+3t^{-1}q^{-5}+4q^{-3}+4tq^{-1}+3t^{2}q \\
 & & & +3t^{3}q^{3}+2t^{4}q^{5}+t^{5}q^{7} \\
  11n27 & \rdE^{2} & 21 & t^{-6}q^{-19}+2t^{-5}q^{-17}+2t^{-4}q^{-15}+t^{-4}q^{-13}+4t^{-3}q^{-13}+3t^{-2}q^{-11} \\
 & & & +3t^{-1}q^{-9}+3q^{-7}+tq^{-5}+t^{2}q^{-3} \\
  & \rdE^{3} & 19 & t^{-6}q^{-19}+2t^{-5}q^{-17}+2t^{-4}q^{-15}+4t^{-3}q^{-13}+2t^{-2}q^{-11}+3t^{-1}q^{-9} \\
 & & & +3q^{-7}+tq^{-5}+t^{2}q^{-3} \\
  11n31 & \rdE^{2} & 29 & t^{-9}q^{-21}+2t^{-8}q^{-19}+2t^{-7}q^{-17}+3t^{-6}q^{-15}+t^{-5}q^{-15}+3t^{-5}q^{-13} \\
 & & & +2t^{-4}q^{-13}+2t^{-4}q^{-11}+2t^{-3}q^{-11}+2t^{-3}q^{-9}+2t^{-2}q^{-9} \\
 & & & +t^{-2}q^{-7}+2t^{-1}q^{-7}+2q^{-5}+tq^{-3}+t^{2}q^{-1} \\
  & \rdE^{3} & 15 & t^{-9}q^{-21}+2t^{-8}q^{-19}+t^{-7}q^{-17}+t^{-6}q^{-15}+2t^{-5}q^{-13}+t^{-4}q^{-11} \\
 & & & +t^{-3}q^{-11}+t^{-3}q^{-9}+t^{-2}q^{-9}+t^{-1}q^{-7}+q^{-5}+tq^{-3}+t^{2}q^{-1} \\
  & \rdE^{4} & 9 & t^{-9}q^{-21}+2t^{-8}q^{-19}+t^{-7}q^{-17}+t^{-5}q^{-13}+t^{-3}q^{-9}+q^{-5} \\
 & & & +tq^{-3}+t^{2}q^{-1} \\
  11n34 & \rdE^{2} & 33 & t^{-5}q^{-9}+2t^{-4}q^{-7}+2t^{-3}q^{-5}+3t^{-2}q^{-3}+t^{-1}q^{-3}+3t^{-1}q^{-1} \\
 & & & +3q^{-1}+2q+2tq+2tq^{3}+3t^{2}q^{3}+t^{2}q^{5}+3t^{3}q^{5} \\
 & & & +2t^{4}q^{7}+2t^{5}q^{9}+t^{6}q^{11} \\
  & \rdE^{3} & 17 & t^{-5}q^{-9}+2t^{-4}q^{-7}+t^{-3}q^{-5}+t^{-2}q^{-3}+2t^{-1}q^{-1}+q^{-1} \\
 & & & +q+tq+2t^{2}q^{3}+t^{3}q^{5}+t^{4}q^{7}+2t^{5}q^{9}+t^{6}q^{11} \\
  & \rdE^{4} & 9 & t^{-5}q^{-9}+2t^{-4}q^{-7}+t^{-3}q^{-5}+q^{-1}+t^{4}q^{7}+2t^{5}q^{9}+t^{6}q^{11} \\
  11n38 & \rdE^{2} & 13 & t^{-6}q^{-11}+t^{-5}q^{-9}+t^{-4}q^{-7}+2t^{-3}q^{-5}+t^{-2}q^{-5}+t^{-2}q^{-3} \\
 & & & +t^{-1}q^{-3}+t^{-1}q^{-1}+q^{-1}+q+tq+t^{2}q^{3} \\
  & \rdE^{3} & 5 & t^{-6}q^{-11}+t^{-5}q^{-9}+t^{-3}q^{-5}+t^{-2}q^{-3}+q^{-1} \\
  & \rdE^{4} & 3 & t^{-6}q^{-11}+t^{-5}q^{-9}+t^{-2}q^{-3} \\
\end{array} \]
\end{scriptsize}
\begin{scriptsize}
\[ \begin{array}{r|ccl}
  \text{Knot $K$} & \rdE^k & \rank \rdE^k & \rdP^k(q, t) = \sum_{i,j} (\rank \rdE^k_{i,j})t^iq^j \\
  \hline
  11n39 & \rdE^{2} & 41 & t^{-5}q^{-9}+t^{-4}q^{-7}+t^{-3}q^{-5}+t^{-2}q^{-5}+2t^{-2}q^{-3}+3t^{-1}q^{-3} \\
 & & & +t^{-1}q^{-1}+5q^{-1}+q+5tq+tq^{3}+6t^{2}q^{3}+5t^{3}q^{5} \\
 & & & +4t^{4}q^{7}+3t^{5}q^{9}+t^{6}q^{11} \\
  & \rdE^{3} & 33 & t^{-5}q^{-9}+t^{-4}q^{-7}+t^{-2}q^{-5}+t^{-2}q^{-3}+2t^{-1}q^{-3}+t^{-1}q^{-1} \\
 & & & +4q^{-1}+5tq+5t^{2}q^{3}+4t^{3}q^{5}+4t^{4}q^{7}+3t^{5}q^{9}+t^{6}q^{11} \\
  & \rdE^{4} & 29 & t^{-5}q^{-9}+t^{-4}q^{-7}+t^{-2}q^{-5}+2t^{-1}q^{-3}+4q^{-1}+4tq+4t^{2}q^{3} \\
 & & & +4t^{3}q^{5}+4t^{4}q^{7}+3t^{5}q^{9}+t^{6}q^{11} \\
  11n42 & \rdE^{2} & 33 & t^{-6}q^{-13}+2t^{-5}q^{-11}+2t^{-4}q^{-9}+3t^{-3}q^{-7}+t^{-2}q^{-7}+3t^{-2}q^{-5} \\
 & & & +2t^{-1}q^{-5}+2t^{-1}q^{-3}+2q^{-3}+3q^{-1}+3tq^{-1}+tq+3t^{2}q \\
 & & & +2t^{3}q^{3}+2t^{4}q^{5}+t^{5}q^{7} \\
  & \rdE^{3} & 17 & t^{-6}q^{-13}+2t^{-5}q^{-11}+t^{-4}q^{-9}+t^{-3}q^{-7}+2t^{-2}q^{-5}+t^{-1}q^{-3} \\
 & & & +q^{-3}+q^{-1}+2tq^{-1}+t^{2}q+t^{3}q^{3}+2t^{4}q^{5}+t^{5}q^{7} \\
  & \rdE^{4} & 9 & t^{-6}q^{-13}+2t^{-5}q^{-11}+t^{-4}q^{-9}+q^{-1}+t^{3}q^{3}+2t^{4}q^{5}+t^{5}q^{7} \\
  11n45 & \rdE^{2} & 41 & t^{-6}q^{-13}+3t^{-5}q^{-11}+4t^{-4}q^{-9}+5t^{-3}q^{-7}+6t^{-2}q^{-5}+t^{-1}q^{-5} \\
 & & & +5t^{-1}q^{-3}+q^{-3}+5q^{-1}+tq^{-1}+3tq+2t^{2}q+t^{2}q^{3} \\
 & & & +t^{3}q^{3}+t^{4}q^{5}+t^{5}q^{7} \\
  & \rdE^{3} & 33 & t^{-6}q^{-13}+3t^{-5}q^{-11}+4t^{-4}q^{-9}+4t^{-3}q^{-7}+5t^{-2}q^{-5}+5t^{-1}q^{-3} \\
 & & & +4q^{-1}+tq^{-1}+2tq+t^{2}q+t^{2}q^{3}+t^{4}q^{5}+t^{5}q^{7} \\
  & \rdE^{4} & 29 & t^{-6}q^{-13}+3t^{-5}q^{-11}+4t^{-4}q^{-9}+4t^{-3}q^{-7}+4t^{-2}q^{-5}+4t^{-1}q^{-3} \\
 & & & +4q^{-1}+2tq+t^{2}q^{3}+t^{4}q^{5}+t^{5}q^{7} \\
  11n49 & \rdE^{2} & 17 & t^{-6}q^{-11}+t^{-5}q^{-9}+t^{-4}q^{-7}+2t^{-3}q^{-5}+t^{-2}q^{-5}+t^{-2}q^{-3} \\
 & & & +t^{-1}q^{-3}+t^{-1}q^{-1}+2q^{-1}+q+2tq+t^{2}q^{3}+t^{3}q^{5} +t^{4}q^{7} \\
  & \rdE^{3} & 9 & t^{-6}q^{-11}+t^{-5}q^{-9}+t^{-3}q^{-5}+t^{-2}q^{-3}+2q^{-1}+tq+t^{3}q^{5} +t^{4}q^{7} \\
  & \rdE^{4} & 5 & t^{-6}q^{-11}+t^{-5}q^{-9}+q^{-1}+t^{3}q^{5}+t^{4}q^{7} \\
  11n57 & \rdE^{2} & 17 & t^{-6}q^{-17}+t^{-5}q^{-17}+t^{-5}q^{-15}+t^{-4}q^{-15}+t^{-4}q^{-13}+2t^{-3}q^{-13} \\
 & & & +t^{-3}q^{-11}+2t^{-2}q^{-11}+t^{-2}q^{-9}+2t^{-1}q^{-9}+2q^{-7}+tq^{-5}+t^{2}q^{-3} \\
  & \rdE^{3} & 7 & t^{-5}q^{-17}+t^{-3}q^{-13}+t^{-2}q^{-11}+t^{-1}q^{-9}+q^{-7}+tq^{-5}+t^{2}q^{-3} \\
  11n61 & \rdE^{2} & 23 & t^{-5}q^{-15}+t^{-4}q^{-13}+t^{-4}q^{-11}+3t^{-3}q^{-11}+3t^{-2}q^{-9}+t^{-2}q^{-7} \\
 & & & +3t^{-1}q^{-7}+t^{-1}q^{-5}+4q^{-5}+2tq^{-3}+2t^{2}q^{-1}+t^{3}q \\
  & \rdE^{3} & 17 & t^{-5}q^{-15}+t^{-4}q^{-13}+3t^{-3}q^{-11}+2t^{-2}q^{-9}+3t^{-1}q^{-7}+3q^{-5} \\
 & & & +tq^{-3}+2t^{2}q^{-1}+t^{3}q \\
  11n67 & \rdE^{2} & 25 & t^{-7}q^{-15}+2t^{-6}q^{-13}+2t^{-5}q^{-11}+3t^{-4}q^{-9}+3t^{-3}q^{-7}+t^{-2}q^{-7} \\
 & & & +2t^{-2}q^{-5}+t^{-1}q^{-5}+2t^{-1}q^{-3}+q^{-3}+2q^{-1}+2tq^{-1} \\
 & & & +t^{2}q+t^{3}q^{3}+t^{4}q^{5} \\
  & \rdE^{3} & 17 & t^{-7}q^{-15}+2t^{-6}q^{-13}+2t^{-5}q^{-11}+2t^{-4}q^{-9}+2t^{-3}q^{-7}+2t^{-2}q^{-5} \\
 & & & +t^{-1}q^{-3}+q^{-3}+q^{-1}+tq^{-1}+t^{3}q^{3}+t^{4}q^{5} \\
  & \rdE^{4} & 13 & t^{-7}q^{-15}+2t^{-6}q^{-13}+2t^{-5}q^{-11}+2t^{-4}q^{-9}+t^{-3}q^{-7}+t^{-2}q^{-5} \\
 & & & +t^{-1}q^{-3}+q^{-1}+t^{3}q^{3}+t^{4}q^{5} \\
  11n70 & \rdE^{2} & 19 & t^{-4}q^{-5}+t^{-3}q^{-3}+2t^{-2}q^{-1}+3t^{-1}q+q+2q^{3}+3tq^{5} \\
 & & & +t^{2}q^{5}+2t^{2}q^{7}+t^{3}q^{7}+t^{3}q^{9}+t^{4}q^{11} \\
  & \rdE^{3} & 13 & t^{-4}q^{-5}+t^{-3}q^{-3}+t^{-2}q^{-1}+3t^{-1}q+q^{3}+2tq^{5}+2t^{2}q^{7} \\
 & & & +t^{3}q^{9}+t^{4}q^{11} \\
  11n73 & \rdE^{2} & 25 & t^{-6}q^{-13}+2t^{-5}q^{-11}+2t^{-4}q^{-9}+3t^{-3}q^{-7}+3t^{-2}q^{-5}+t^{-1}q^{-5} \\
 & & & +2t^{-1}q^{-3}+q^{-3}+3q^{-1}+tq^{-1}+tq+2t^{2}q+t^{3}q^{3} \\
 & & & +t^{4}q^{5}+t^{5}q^{7} \\
  & \rdE^{3} & 17 & t^{-6}q^{-13}+2t^{-5}q^{-11}+2t^{-4}q^{-9}+2t^{-3}q^{-7}+2t^{-2}q^{-5}+2t^{-1}q^{-3} \\
 & & & +2q^{-1}+tq^{-1}+t^{2}q+t^{4}q^{5}+t^{5}q^{7} \\
  & \rdE^{4} & 13 & t^{-6}q^{-13}+2t^{-5}q^{-11}+2t^{-4}q^{-9}+2t^{-3}q^{-7}+t^{-2}q^{-5}+t^{-1}q^{-3} \\
 & & & +2q^{-1}+t^{4}q^{5}+t^{5}q^{7} \\
\end{array} \]
\end{scriptsize}
\begin{scriptsize}
\[ \begin{array}{r|ccl}
  \text{Knot $K$} & \rdE^k & \rank \rdE^k & \rdP^k(q, t) = \sum_{i,j} (\rank \rdE^k_{i,j})t^iq^j \\
  \hline
  11n74 & \rdE^{2} & 25 & t^{-5}q^{-9}+t^{-4}q^{-7}+t^{-3}q^{-5}+2t^{-2}q^{-3}+t^{-1}q^{-3}+t^{-1}q^{-1} \\
 & & & +3q^{-1}+q+2tq+tq^{3}+3t^{2}q^{3}+3t^{3}q^{5}+2t^{4}q^{7}+2t^{5}q^{9} +t^{6}q^{11} \\
  & \rdE^{3} & 17 & t^{-5}q^{-9}+t^{-4}q^{-7}+t^{-2}q^{-3}+t^{-1}q^{-1}+2q^{-1}+2tq+2t^{2}q^{3} \\
 & & & +2t^{3}q^{5}+2t^{4}q^{7}+2t^{5}q^{9}+t^{6}q^{11} \\
  & \rdE^{4} & 13 & t^{-5}q^{-9}+t^{-4}q^{-7}+2q^{-1}+tq+t^{2}q^{3}+2t^{3}q^{5}+2t^{4}q^{7} \\
 & & & +2t^{5}q^{9}+t^{6}q^{11} \\
  11n77 & \rdE^{2} & 35 & t^{-11}q^{-29}+3t^{-10}q^{-27}+4t^{-9}q^{-25}+5t^{-8}q^{-23}+6t^{-7}q^{-21}+5t^{-6}q^{-19} \\
 & & & +t^{-5}q^{-19}+4t^{-5}q^{-17}+3t^{-4}q^{-15}+t^{-3}q^{-15}+t^{-2}q^{-13}+q^{-9} \\
  & \rdE^{3} & 31 & t^{-11}q^{-29}+3t^{-10}q^{-27}+4t^{-9}q^{-25}+5t^{-8}q^{-23}+5t^{-7}q^{-21}+5t^{-6}q^{-19} \\
 & & & +4t^{-5}q^{-17}+2t^{-4}q^{-15}+t^{-3}q^{-15}+q^{-9} \\
  & \rdE^{4} & 29 & t^{-11}q^{-29}+3t^{-10}q^{-27}+4t^{-9}q^{-25}+5t^{-8}q^{-23}+5t^{-7}q^{-21}+4t^{-6}q^{-19} \\
 & & & +4t^{-5}q^{-17}+2t^{-4}q^{-15}+q^{-9} \\
  11n79 & \rdE^{2} & 17 & t^{-4}q^{-7}+t^{-3}q^{-5}+2t^{-2}q^{-3}+3t^{-1}q^{-1}+q^{-1}+2q+3tq^{3} \\
 & & & +2t^{2}q^{5}+t^{3}q^{7}+t^{4}q^{9} \\
  & \rdE^{3} & 15 & t^{-4}q^{-7}+t^{-3}q^{-5}+t^{-2}q^{-3}+3t^{-1}q^{-1}+2q+3tq^{3}+2t^{2}q^{5} \\
 & & & +t^{3}q^{7}+t^{4}q^{9} \\
  11n80 & \rdE^{2} & 31 & t^{-7}q^{-17}+2t^{-6}q^{-15}+3t^{-5}q^{-13}+4t^{-4}q^{-11}+4t^{-3}q^{-9}+t^{-2}q^{-9} \\
 & & & +4t^{-2}q^{-7}+t^{-1}q^{-7}+3t^{-1}q^{-5}+q^{-5}+2q^{-3}+2tq^{-3} \\
 & & & +t^{2}q^{-1}+t^{3}q+t^{4}q^{3} \\
  & \rdE^{3} & 23 & t^{-7}q^{-17}+2t^{-6}q^{-15}+3t^{-5}q^{-13}+3t^{-4}q^{-11}+3t^{-3}q^{-9}+4t^{-2}q^{-7} \\
 & & & +2t^{-1}q^{-5}+q^{-5}+q^{-3}+tq^{-3}+t^{3}q+t^{4}q^{3} \\
  & \rdE^{4} & 19 & t^{-7}q^{-17}+2t^{-6}q^{-15}+3t^{-5}q^{-13}+3t^{-4}q^{-11}+2t^{-3}q^{-9}+3t^{-2}q^{-7} \\
 & & & +2t^{-1}q^{-5}+q^{-3}+t^{3}q+t^{4}q^{3} \\
  11n81 & \rdE^{2} & 29 & 2t^{-6}q^{-19}+3t^{-5}q^{-17}+3t^{-4}q^{-15}+t^{-4}q^{-13}+6t^{-3}q^{-13}+4t^{-2}q^{-11} \\
 & & & +4t^{-1}q^{-9}+4q^{-7}+tq^{-5}+t^{2}q^{-3} \\
  & \rdE^{3} & 27 & 2t^{-6}q^{-19}+3t^{-5}q^{-17}+3t^{-4}q^{-15}+6t^{-3}q^{-13}+3t^{-2}q^{-11}+4t^{-1}q^{-9} \\
 & & & +4q^{-7}+tq^{-5}+t^{2}q^{-3} \\
  11n88 & \rdE^{2} & 13 & t^{-5}q^{-17}+t^{-4}q^{-15}+t^{-4}q^{-13}+2t^{-3}q^{-13}+2t^{-2}q^{-11}+2t^{-1}q^{-9} \\
 & & & +2q^{-7}+tq^{-5}+t^{2}q^{-3} \\
  & \rdE^{3} & 11 & t^{-5}q^{-17}+t^{-4}q^{-15}+2t^{-3}q^{-13}+t^{-2}q^{-11}+2t^{-1}q^{-9}+2q^{-7} \\
 & & & +tq^{-5}+t^{2}q^{-3} \\
  11n92 & \rdE^{2} & 17 & t^{-5}q^{-9}+2t^{-4}q^{-7}+2t^{-3}q^{-5}+3t^{-2}q^{-3}+3t^{-1}q^{-1}+q^{-1} \\
 & & & +2q+2tq^{3}+t^{2}q^{5} \\
  & \rdE^{3} & 15 & t^{-5}q^{-9}+2t^{-4}q^{-7}+2t^{-3}q^{-5}+2t^{-2}q^{-3}+3t^{-1}q^{-1}+2q \\
 & & & +2tq^{3}+t^{2}q^{5} \\
  11n96 & \rdE^{2} & 25 & t^{-5}q^{-9}+2t^{-4}q^{-7}+2t^{-3}q^{-5}+3t^{-2}q^{-3}+t^{-1}q^{-3}+3t^{-1}q^{-1} \\
 & & & +2q^{-1}+2q+tq+2tq^{3}+2t^{2}q^{3}+t^{2}q^{5}+t^{3}q^{5}+t^{4}q^{7} \\
 & & & +t^{5}q^{9} \\
  & \rdE^{3} & 15 & t^{-5}q^{-9}+2t^{-4}q^{-7}+t^{-3}q^{-5}+t^{-2}q^{-3}+3t^{-1}q^{-1}+q \\
 & & & +tq+tq^{3}+t^{2}q^{3}+t^{2}q^{5}+t^{4}q^{7}+t^{5}q^{9} \\
  & \rdE^{4} & 11 & t^{-5}q^{-9}+2t^{-4}q^{-7}+t^{-3}q^{-5}+2t^{-1}q^{-1}+q+tq^{3}+t^{2}q^{5} \\
 & & & +t^{4}q^{7}+t^{5}q^{9} \\
  11n97 & \rdE^{2} & 25 & t^{-5}q^{-11}+2t^{-4}q^{-9}+2t^{-3}q^{-7}+3t^{-2}q^{-5}+3t^{-1}q^{-3}+q^{-3} \\
 & & & +3q^{-1}+tq^{-1}+2tq+t^{2}q+t^{2}q^{3}+2t^{3}q^{3}+t^{4}q^{5} \\
 & & & +t^{5}q^{7}+t^{6}q^{9} \\
  & \rdE^{3} & 17 & t^{-5}q^{-11}+2t^{-4}q^{-9}+2t^{-3}q^{-7}+2t^{-2}q^{-5}+2t^{-1}q^{-3}+3q^{-1} \\
 & & & +tq+t^{2}q+t^{3}q^{3}+t^{5}q^{7}+t^{6}q^{9} \\
  & \rdE^{4} & 13 & t^{-5}q^{-11}+2t^{-4}q^{-9}+2t^{-3}q^{-7}+2t^{-2}q^{-5}+t^{-1}q^{-3}+2q^{-1} \\
 & & & +tq+t^{5}q^{7}+t^{6}q^{9} \\
\end{array} \]
\end{scriptsize}
\begin{scriptsize}
\[ \begin{array}{r|ccl}
  \text{Knot $K$} & \rdE^k & \rank \rdE^k & \rdP^k(q, t) = \sum_{i,j} (\rank \rdE^k_{i,j})t^iq^j \\
  \hline
  11n102 & \rdE^{2} & 19 & t^{-8}q^{-17}+t^{-7}q^{-15}+t^{-6}q^{-13}+2t^{-5}q^{-11}+t^{-4}q^{-11}+t^{-4}q^{-9} \\
 & & & +2t^{-3}q^{-9}+t^{-3}q^{-7}+2t^{-2}q^{-7}+t^{-2}q^{-5}+2t^{-1}q^{-5} \\
 & & & +2q^{-3}+tq^{-1}+t^{2}q \\
  & \rdE^{3} & 11 & t^{-8}q^{-17}+t^{-7}q^{-15}+t^{-5}q^{-11}+t^{-4}q^{-9}+t^{-3}q^{-9}+2t^{-2}q^{-7} \\
 & & & +t^{-1}q^{-5}+q^{-3}+tq^{-1}+t^{2}q \\
  & \rdE^{4} & 7 & t^{-8}q^{-17}+t^{-7}q^{-15}+t^{-3}q^{-9}+t^{-2}q^{-7}+q^{-3}+tq^{-1}+t^{2}q \\
  11n104 & \rdE^{2} & 21 & t^{-8}q^{-21}+t^{-7}q^{-19}+t^{-6}q^{-17}+t^{-5}q^{-17}+2t^{-5}q^{-15}+t^{-4}q^{-15} \\
 & & & +2t^{-4}q^{-13}+2t^{-3}q^{-13}+t^{-3}q^{-11}+2t^{-2}q^{-11}+t^{-2}q^{-9} \\
 & & & +2t^{-1}q^{-9}+2q^{-7}+tq^{-5}+t^{2}q^{-3} \\
  & \rdE^{3} & 11 & t^{-8}q^{-21}+t^{-7}q^{-19}+t^{-5}q^{-17}+t^{-5}q^{-15}+t^{-4}q^{-13}+t^{-3}q^{-13} \\
 & & & +t^{-2}q^{-11}+t^{-1}q^{-9}+q^{-7}+tq^{-5}+t^{2}q^{-3} \\
  & \rdE^{4} & 7 & t^{-8}q^{-21}+t^{-7}q^{-19}+t^{-5}q^{-17}+t^{-3}q^{-13}+q^{-7}+tq^{-5}+t^{2}q^{-3} \\
  11n111 & \rdE^{2} & 23 & t^{-4}q^{-5}+t^{-3}q^{-3}+t^{-2}q^{-1}+2t^{-1}q+2q+q^{3}+2tq^{3} \\
 & & & +tq^{5}+3t^{2}q^{5}+t^{2}q^{7}+3t^{3}q^{7}+2t^{4}q^{9}+2t^{5}q^{11}+t^{6}q^{13} \\
  & \rdE^{3} & 15 & t^{-4}q^{-5}+t^{-3}q^{-3}+t^{-1}q+q+q^{3}+tq^{3}+3t^{2}q^{5}+2t^{3}q^{7} \\
 & & & +t^{4}q^{9}+2t^{5}q^{11}+t^{6}q^{13} \\
  & \rdE^{4} & 11 & t^{-4}q^{-5}+t^{-3}q^{-3}+q+tq^{3}+2t^{2}q^{5}+t^{3}q^{7}+t^{4}q^{9}+2t^{5}q^{11}+t^{6}q^{13} \\
  11n116 & \rdE^{2} & 17 & t^{-6}q^{-13}+t^{-5}q^{-11}+t^{-4}q^{-9}+2t^{-3}q^{-7}+t^{-2}q^{-7}+t^{-2}q^{-5} \\
 & & & +t^{-1}q^{-5}+t^{-1}q^{-3}+q^{-3}+2q^{-1}+2tq^{-1}+t^{2}q+t^{3}q^{3}+t^{4}q^{5} \\
  & \rdE^{3} & 9 & t^{-6}q^{-13}+t^{-5}q^{-11}+t^{-3}q^{-7}+t^{-2}q^{-5}+q^{-3}+q^{-1}+tq^{-1} \\
 & & & +t^{3}q^{3}+t^{4}q^{5} \\
  & \rdE^{4} & 5 & t^{-6}q^{-13}+t^{-5}q^{-11}+q^{-1}+t^{3}q^{3}+t^{4}q^{5} \\
  11n126 & \rdE^{2} & 29 & t^{-8}q^{-23}+3t^{-7}q^{-21}+3t^{-6}q^{-19}+5t^{-5}q^{-17}+5t^{-4}q^{-15}+t^{-4}q^{-13} \\
 & & & +4t^{-3}q^{-13}+4t^{-2}q^{-11}+2t^{-1}q^{-9}+q^{-7} \\
  & \rdE^{3} & 27 & t^{-8}q^{-23}+3t^{-7}q^{-21}+3t^{-6}q^{-19}+5t^{-5}q^{-17}+5t^{-4}q^{-15}+4t^{-3}q^{-13} \\
 & & & +3t^{-2}q^{-11}+2t^{-1}q^{-9}+q^{-7} \\
  11n133 & \rdE^{2} & 31 & t^{-5}q^{-15}+2t^{-4}q^{-13}+t^{-4}q^{-11}+4t^{-3}q^{-11}+4t^{-2}q^{-9}+t^{-2}q^{-7} \\
 & & & +5t^{-1}q^{-7}+t^{-1}q^{-5}+5q^{-5}+3tq^{-3}+3t^{2}q^{-1}+t^{3}q \\
  & \rdE^{3} & 25 & t^{-5}q^{-15}+2t^{-4}q^{-13}+4t^{-3}q^{-11}+3t^{-2}q^{-9}+5t^{-1}q^{-7}+4q^{-5} \\
 & & & +2tq^{-3}+3t^{2}q^{-1}+t^{3}q \\
  11n135 & \rdE^{2} & 21 & t^{-8}q^{-19}+t^{-7}q^{-17}+t^{-6}q^{-15}+t^{-5}q^{-15}+2t^{-5}q^{-13}+2t^{-4}q^{-13} \\
 & & & +t^{-4}q^{-11}+2t^{-3}q^{-11}+t^{-3}q^{-9}+2t^{-2}q^{-9}+t^{-2}q^{-7}+2t^{-1}q^{-7} \\
 & & & +2q^{-5}+tq^{-3}+t^{2}q^{-1} \\
  & \rdE^{3} & 13 & t^{-8}q^{-19}+t^{-7}q^{-17}+t^{-5}q^{-15}+t^{-5}q^{-13}+t^{-4}q^{-13}+t^{-4}q^{-11} \\
 & & & +t^{-3}q^{-11}+2t^{-2}q^{-9}+t^{-1}q^{-7}+q^{-5}+tq^{-3}+t^{2}q^{-1} \\
  & \rdE^{4} & 9 & t^{-8}q^{-19}+t^{-7}q^{-17}+t^{-5}q^{-15}+t^{-4}q^{-13}+t^{-3}q^{-11}+t^{-2}q^{-9} \\
 & & & +q^{-5}+tq^{-3}+t^{2}q^{-1} \\
  11n138 & \rdE^{2} & 17 & t^{-6}q^{-11}+t^{-5}q^{-9}+2t^{-4}q^{-7}+3t^{-3}q^{-5}+2t^{-2}q^{-3}+3t^{-1}q^{-1} \\
 & & & +q^{-1}+2q+tq^{3}+t^{2}q^{5} \\
  & \rdE^{3} & 15 & t^{-6}q^{-11}+t^{-5}q^{-9}+2t^{-4}q^{-7}+3t^{-3}q^{-5}+t^{-2}q^{-3}+3t^{-1}q^{-1} \\
 & & & +2q+tq^{3}+t^{2}q^{5} \\
  11n143 & \rdE^{2} & 25 & t^{-6}q^{-13}+2t^{-5}q^{-11}+2t^{-4}q^{-9}+3t^{-3}q^{-7}+t^{-2}q^{-7}+3t^{-2}q^{-5} \\
 & & & +t^{-1}q^{-5}+2t^{-1}q^{-3}+q^{-3}+3q^{-1}+2tq^{-1}+tq+t^{2}q \\
 & & & +t^{3}q^{3}+t^{4}q^{5} \\
  & \rdE^{3} & 17 & t^{-6}q^{-13}+2t^{-5}q^{-11}+t^{-4}q^{-9}+2t^{-3}q^{-7}+3t^{-2}q^{-5}+t^{-1}q^{-3} \\
 & & & +q^{-3}+2q^{-1}+tq^{-1}+tq+t^{3}q^{3}+t^{4}q^{5} \\
  & \rdE^{4} & 13 & t^{-6}q^{-13}+2t^{-5}q^{-11}+t^{-4}q^{-9}+t^{-3}q^{-7}+2t^{-2}q^{-5}+t^{-1}q^{-3} \\
 & & & +2q^{-1}+tq+t^{3}q^{3}+t^{4}q^{5} \\
\end{array} \]
\end{scriptsize}
\begin{scriptsize}
\[ \begin{array}{r|ccl}
  \text{Knot $K$} & \rdE^k & \rank \rdE^k & \rdP^k(q, t) = \sum_{i,j} (\rank \rdE^k_{i,j})t^iq^j \\
  \hline
  11n145 & \rdE^{2} & 25 & t^{-5}q^{-11}+2t^{-4}q^{-9}+2t^{-3}q^{-7}+3t^{-2}q^{-5}+t^{-1}q^{-5}+3t^{-1}q^{-3} \\
 & & & +q^{-3}+3q^{-1}+tq^{-1}+2tq+2t^{2}q+t^{2}q^{3}+t^{3}q^{3}+t^{4}q^{5}+t^{5}q^{7} \\
  & \rdE^{3} & 17 & t^{-5}q^{-11}+2t^{-4}q^{-9}+t^{-3}q^{-7}+2t^{-2}q^{-5}+3t^{-1}q^{-3}+2q^{-1} \\
 & & & +tq^{-1}+tq+t^{2}q+t^{2}q^{3}+t^{4}q^{5}+t^{5}q^{7} \\
  & \rdE^{4} & 13 & t^{-5}q^{-11}+2t^{-4}q^{-9}+t^{-3}q^{-7}+t^{-2}q^{-5}+2t^{-1}q^{-3}+2q^{-1} \\
 & & & +tq+t^{2}q^{3}+t^{4}q^{5}+t^{5}q^{7} \\
  11n151 & \rdE^{2} & 39 & t^{-4}q^{-5}+t^{-3}q^{-3}+t^{-2}q^{-1}+2t^{-1}q+3q+q^{3}+4tq^{3} \\
 & & & +tq^{5}+5t^{2}q^{5}+t^{2}q^{7}+6t^{3}q^{7}+5t^{4}q^{9}+4t^{5}q^{11}+3t^{6}q^{13}+t^{7}q^{15} \\
  & \rdE^{3} & 31 & t^{-4}q^{-5}+t^{-3}q^{-3}+t^{-1}q+2q+q^{3}+3tq^{3}+5t^{2}q^{5}+5t^{3}q^{7} \\
 & & & +4t^{4}q^{9}+4t^{5}q^{11}+3t^{6}q^{13}+t^{7}q^{15} \\
  & \rdE^{4} & 27 & t^{-4}q^{-5}+t^{-3}q^{-3}+2q+3tq^{3}+4t^{2}q^{5}+4t^{3}q^{7}+4t^{4}q^{9} \\
 & & & +4t^{5}q^{11}+3t^{6}q^{13}+t^{7}q^{15} \\
  11n152 & \rdE^{2} & 39 & t^{-7}q^{-17}+3t^{-6}q^{-15}+4t^{-5}q^{-13}+5t^{-4}q^{-11}+6t^{-3}q^{-9}+t^{-2}q^{-9} \\
 & & & +5t^{-2}q^{-7}+t^{-1}q^{-7}+4t^{-1}q^{-5}+q^{-5}+3q^{-3}+2tq^{-3} \\
 & & & +t^{2}q^{-1}+t^{3}q+t^{4}q^{3} \\
  & \rdE^{3} & 31 & t^{-7}q^{-17}+3t^{-6}q^{-15}+4t^{-5}q^{-13}+4t^{-4}q^{-11}+5t^{-3}q^{-9}+5t^{-2}q^{-7} \\
 & & & +3t^{-1}q^{-5}+q^{-5}+2q^{-3}+tq^{-3}+t^{3}q+t^{4}q^{3} \\
  & \rdE^{4} & 27 & t^{-7}q^{-17}+3t^{-6}q^{-15}+4t^{-5}q^{-13}+4t^{-4}q^{-11}+4t^{-3}q^{-9}+4t^{-2}q^{-7} \\
 & & & +3t^{-1}q^{-5}+2q^{-3}+t^{3}q+t^{4}q^{3} \\
  11n183 & \rdE^{2} & 29 & q^{5}+t^{2}q^{9}+2t^{3}q^{9}+t^{3}q^{11}+4t^{4}q^{11}+3t^{5}q^{13}+t^{5}q^{15}+5t^{6}q^{15} \\
 & & & +4t^{7}q^{17}+3t^{8}q^{19}+3t^{9}q^{21}+t^{10}q^{23} \\
  & \rdE^{3} & 25 & q^{5}+2t^{3}q^{9}+t^{3}q^{11}+3t^{4}q^{11}+3t^{5}q^{13}+5t^{6}q^{15}+3t^{7}q^{17} \\
 & & & +3t^{8}q^{19}+3t^{9}q^{21}+t^{10}q^{23} \\
  & \rdE^{4} & 23 & q^{5}+2t^{3}q^{9}+3t^{4}q^{11}+3t^{5}q^{13}+4t^{6}q^{15}+3t^{7}q^{17}+3t^{8}q^{19} \\
 & & & +3t^{9}q^{21}+t^{10}q^{23} \\
  T(3,3) & \rdE^{2} & 6 & t^{-4}q^{-9}+t^{-2}q^{-5}+t^{-1}q^{-3}+q^{-3}+2q^{-1} \\
  & \rdE^{3} & 4 & t^{-4}q^{-9}+t^{-1}q^{-3}+2q^{-1} \\
  T(3,4) & \rdE^{2} & 5 & q^{5}+t^{2}q^{9}+t^{3}q^{11}+t^{4}q^{11}+t^{5}q^{15} \\
  & \rdE^{3} & 3 & q^{5}+t^{3}q^{11}+t^{5}q^{15} \\
  T(4,4) & \rdE^{2} & 12 & t^{-6}q^{-10}+t^{-4}q^{-6}+t^{-3}q^{-4}+t^{-2}q^{-4}+t^{-1}+q^{-2}+3 \\
 & & & +t^{2}q^{2}+2t^{2}q^{4} \\
  & \rdE^{3} & 10 & t^{-6}q^{-10}+t^{-3}q^{-4}+t^{-1}+q^{-2}+3+t^{2}q^{2}+2t^{2}q^{4} \\
  & \rdE^{4} & 8 & t^{-6}q^{-10}+t^{-1}+q^{-2}+2+t^{2}q^{2}+2t^{2}q^{4} \\
  T(3,5) & \rdE^{2} & 7 & q^{7}+t^{2}q^{11}+t^{3}q^{13}+t^{4}q^{13}+t^{5}q^{17}+t^{6}q^{17}+t^{7}q^{19} \\
  & \rdE^{3} & 3 & q^{7}+t^{3}q^{13}+t^{6}q^{17} \\
  & \rdE^{4} & 1 & q^{7} \\
  T(4,5) & \rdE^{2} & 13 & q^{11}+t^{2}q^{15}+t^{3}q^{17}+t^{4}q^{17}+t^{5}q^{21}+t^{6}q^{19}+t^{6}q^{21}+t^{7}q^{21} \\
 & & & +t^{7}q^{23}+t^{8}q^{23}+t^{9}q^{25}+t^{9}q^{27}+t^{10}q^{27} \\
  & \rdE^{3} & 9 & q^{11}+t^{3}q^{17}+t^{6}q^{19}+t^{6}q^{21}+t^{7}q^{21}+t^{8}q^{23}+t^{9}q^{25}+t^{9}q^{27} \\
 & & & +t^{10}q^{27} \\
  & \rdE^{4} & 7 & q^{11}+t^{6}q^{19}+t^{7}q^{21}+t^{8}q^{23}+t^{9}q^{25}+t^{9}q^{27}+t^{10}q^{27} \\
  T(3,6) & \rdE^{2} & 10 & t^{-8}q^{-15}+t^{-6}q^{-11}+t^{-5}q^{-9}+t^{-4}q^{-9}+t^{-3}q^{-5}+t^{-2}q^{-5} \\
 & & & +t^{-1}q^{-3}+q^{-3}+2q^{-1} \\
  & \rdE^{3} & 6 & t^{-8}q^{-15}+t^{-5}q^{-9}+t^{-2}q^{-5}+q^{-3}+2q^{-1} \\
  & \rdE^{4} & 4 & t^{-8}q^{-15}+q^{-3}+2q^{-1} \\
  T(3,7) & \rdE^{2} & 9 & q^{11}+t^{2}q^{15}+t^{3}q^{17}+t^{4}q^{17}+t^{5}q^{21}+t^{6}q^{21}+t^{7}q^{23}+t^{8}q^{23} \\
 & & & +t^{9}q^{27} \\
  & \rdE^{3} & 5 & q^{11}+t^{3}q^{17}+t^{6}q^{21}+t^{8}q^{23}+t^{9}q^{27} \\
  & \rdE^{4} & 3 & q^{11}+t^{8}q^{23}+t^{9}q^{27} \\
\end{array} \]
\end{scriptsize}
\begin{scriptsize}
\[ \begin{array}{r|ccl}
  \text{Knot $K$} & \rdE^k & \rank \rdE^k & \rdP^k(q, t) = \sum_{i,j} (\rank \rdE^k_{i,j})t^iq^j \\
  \hline
  T(3,8) & \rdE^{2} & 11 & q^{13}+t^{2}q^{17}+t^{3}q^{19}+t^{4}q^{19}+t^{5}q^{23}+t^{6}q^{23}+t^{7}q^{25}+t^{8}q^{25} \\
 & & & +t^{9}q^{29}+t^{10}q^{29}+t^{11}q^{31} \\
  & \rdE^{3} & 7 & q^{13}+t^{3}q^{19}+t^{6}q^{23}+t^{8}q^{25}+t^{9}q^{29}+t^{10}q^{29}+t^{11}q^{31} \\
  & \rdE^{4} & 5 & q^{13}+t^{8}q^{25}+t^{9}q^{29}+t^{10}q^{29}+t^{11}q^{31} \\
\end{array} \]
\end{scriptsize}


\begin{thebibliography}{9}
  
\bibitem{baldwin} J. Baldwin, {\it On the Spectral Sequence from
  Khovanov Homology to Heegaard Floer Homology}, arXiv:0809.3293.
  
\bibitem{bncat} D. Bar-Natan, {\it On Khovanov's Categorification of
  the Jones Polynomial}, Algebraic and Geometric Topology, {\bf 2}
  (2002), 337--370.
  
\bibitem{bntangle} D. Bar-Natan, {\it Khovanov's Homology for Tangles
  and Cobordisms}, Geometry and Topology {\bf 9} (2005), 1443--1499,
  arXiv:math.GT/0410495.
  
\bibitem{bnfast} D. Bar-Natan, {\it Fast Khovanov Homology
  Computations}, Journal of Knot Theory and Its Ramifications, {\bf
  16--3} (2007), 243--255, arXiv:math.GT/0606318.
  
\bibitem{bnknotth} {\it The Mathematica Package {\tt KnotTheory`}},
  \url{http://katlas.math.toronto.edu/wiki/The_Mathematica_Package_KnotTheory`}.

\bibitem{bloomss} J. Bloom, {\it A link Surgery Spectral Sequence in
  Monopole Floer Homology}, arXiv:0909.0816.
  
\bibitem{bloomodd} J. Bloom, {\it Odd Khovanov Homology is Mutation
  Invariant}, arXiv:0903.3746.
  
\bibitem{snappy} M. Culler, N. M. Dunfield, and J. R. Weeks, {\it
  SnapPy, A Computer Program for Studying the Geometry and Topology of
  3-manifolds}, \url{http://snappy.computop.org}.
  
\bibitem{freedmanetal} M. Freedman, R. Gompf, S. Morrison and
  K. Walker, {\it Man and Machine Thinking about the Smooth
    4-dimensional Poincar\'e Conjecture}, arXiv:0906.5177.
  
\bibitem{htwknots} J. Hoste, M. Thistlethwaite, J. Weeks, {\it The
  First 1,701,935 Knots}, Math. Intelligencer, {\bf 20} (1998),
  33--48.
  
\bibitem{kronmorwss} P. Kronheimer and T. Mrowka, {\it Khovanov
  Homology is an Unknot Detector}, arXiv:1005.4346.
  
\bibitem{khovanov} M. Khovanov, {\it A Categorification of the Jones
  Polynomial}, Duke Math. J. {\bf 101} (2000), 359--426.
  
\bibitem{lotbordered} R. Lipshitz, P. Ozsv\'ath, and D. Thurston, {\it
  Bordered Heegaard Floer Homology: Invariance and Pairing},
  arXiv:0810.0687.

\bibitem{lotslicing} R. Lipshitz, P. Ozsv\'ath, and D. Thurston, {\it
  Slicing Planar Grid Diagrams: A Gentle Introduction to Bordered
  Heegaard Floer Homology}, arXiv:0810.0695v2.

\bibitem{lothfhat} R. Lipshitz, P. Ozsv\'ath, and D. Thurston, {\it
  Computing $\HFhat$ by Factoring Mapping Classes}, arXiv:1010.2550v1.
  
\bibitem{lotss1} R. Lipshitz, P. Ozsv\'ath, and D. Thurston, {\it
  Bordered Floer Homology and the Spectral Sequence of a Branched
  Double Cover I}, arXiv:1011.0499v1.
  
\bibitem{lotss2} R. Lipshitz, P. Ozsv\'ath, and D. Thurston, {\it
  Bordered Floer Homology and the Spectral Sequence of a Branched
  Double Cover II: The Spectral Sequences Agree}, in preparation.
  
\bibitem{lotbordprog} R. Lipshitz, P. Ozsv\'ath, and D. Thurston,
  BordProg,
  \url{http://www.math.columbia.edu/~lipshitz/research.html}.

\bibitem{mos} C. Manolescu and Peter Ozsv\'ath, {\it On the Khovanov
  and Knot Floer Homologies of Quasi-alternating Links},
  arXiv:0708.3249.

\bibitem{oszdouble} P. Ozsv\'ath and Z. Szab\'o, {\it On the Heegaard
  Floer Homology of Branched Double-covers}, Adv. Math. {\bf 194},
  (2005), 1--33.

\bibitem{sage} William A. Stein et al., {\it Sage Mathematics
  Software}, The Sage Development Team, \url{http://www.sagemath.org}.

\bibitem{shumakovitch} A. Shumakovitch, {\it Torsion of the Khovanov
  Homology}, arXiv:math/0405474v1.
  
\bibitem{stoimenow} A. Stoimenow, {\it Knot Data Tables},
  \url{http://stoimenov.net/stoimeno/homepage/ptab/index.html}.
  
\bibitem{szaboss} Z. Szab\'o.  {\it A Geometric Spectral Sequence in
  Khovanov Homology}, arXiv:1010.4252v1.

\bibitem{szabossodd} Z. Szab\'o, {\it A Geometric Spectral Sequence in
  Odd Khovanov Homology}, in preparation.
  
\bibitem{zhan} B. Zhan, personal communication.

\end{thebibliography}
\end{document}